\documentclass[12pt]{amsart}

\title{Complexes of modules over hierarchical groups}
\author{Ioannis Emmanouil and Olympia Talelli}
\thanks{Research funded by the Hellenic Foundation for Research 
and Innovation (H.F.R.I.) under the "3rd Call for H.F.R.I.\ 
Research Projects to Support Faculty Members and Researchers", 
project number 24921}

\newtheorem{Lemma}{Lemma}[section]
\newtheorem{Proposition}[Lemma]{Proposition}
\newtheorem{Theorem}[Lemma]{Theorem}
\newtheorem{Corollary}[Lemma]{Corollary}

\begin{document}

\begin{abstract}
Let $\overline{\mathfrak{F}}$ be the group class obtained from 
the class $\mathfrak{F}$ of finite groups by iterated applications 
of Kropholler's operation ${\scriptstyle{{\bf LH}}}$ and Talelli's 
operation $\Phi$. We describe the structure of the cokenels of 
acyclic complexes of projective modules over the group algebra of 
$\overline{\mathfrak{F}}$-groups with coefficients in a commutative 
ring. If $G$ is an $\overline{\mathfrak{F}}$-group, we show that 
any acyclic complex of projective (respectively, flat) 
$\mathbb{Q}G$-modules is contractible (respectively, pure acyclic). 
If $G$ is torsion-free, the same conclusions hold for acyclic 
complexes of projective or flat $\mathbb{Z}G$-modules. Analogous 
results are valid for acyclic complexes of injective modules. We
present some applications regarding modules that admit complete 
resolutions, groups with periodic cohomology after some steps and 
the relation between Gorenstein and ordinary homological dimensions
for modules over group algebras.
\end{abstract}

\maketitle
\tableofcontents

\addtocounter{section}{-1}
\section{Introduction}

\noindent
Given a pointed space $X$, its loop space $\Omega X$ consists of all 
continuous paths that start and end at the base point. Iterating this 
construction, we obtain the higher loop spaces $(\Omega^nX)_{n \geq 0}$
of $X$, where we set $\Omega^0X=X$. An infinite loop space is a pointed 
space $X$ that can be infinitely de-looped, in the following sense: There 
is a sequence of pointed spaces $(X_n)_{n \leq 0}$, such that $X_0=X$ and 
each $X_n$ is homotopy equivalent to the loop space of $X_{n-1}$, so that 
$X \simeq \Omega X_{-1}$, $X_{-1} \simeq \Omega X_{-2}, \ldots$; see, 
for example, \cite{May}.

There is an analogue of this homotopy-theoretic situation in algebra.
For any module $M$ over a ring $R$, we can define the sequence of 
syzygy modules $(\Omega^nM)_{n \geq 0}$. Here, $\Omega^0M=M$ and the 
modules $\Omega^nM$, $n \geq 0$, are defined as the cokernels in a 
projective resolution of $M$; they are uniquely determined by $M$, 
up to addition of a projective module. The algebraic analogue of 
an infinite loop space is a module $M$, that appears as a cokernel 
of a doubly infinite acyclic complex ${\bf P}$ of projective modules. 
In that case, there is a sequence of modules $(M_n)_{n \leq 0}$ (the 
sequence of cokernels of ${\bf P}$ to the right of $M$), so that 
$M_0=M$ and $M_n = \Omega^i M_{n-i}$ for all $n \leq 0$ and $i \geq 1$.
Such doubly infinite complexes appear in group theory (Tate cohomology, 
cohomological periodicity, free actions on finite dimensional homotopy 
spheres), representation theory (maximal Cohen-Macaulay modules) and 
algebraic geometry (Grothendieck duality, singularity theory).

Let ${\mathcal P}(R)$ be the class of modules that appear as 
cokernels of acyclic complexes of projective modules. Not all 
modules are expected to be ${\mathcal P}(R)$-modules, since such 
modules are necessarily isomorphic to submodules of free modules. 
In fact, a necessary and sufficient condition for all modules to be 
contained in ${\mathcal P}(R)$ is that all modules be embeddable into 
free modules. The Faith-Walker theorem \cite{FW} implies  that the 
latter condition is equivalent to the ring $R$ being quasi-Frobenius
(as it implies that all injective modules are projective).

On the other hand, all projective modules are contained in 
${\mathcal P}(R)$. One could ask whether any ${\mathcal P}(R)$-module 
is necessarily projective, i.e.\ whether any acyclic complex of 
projective modules is necessarily contractible. If all cokernels of 
an acyclic complex of projective modules are flat, then these cokernels 
are projective; cf.\ \cite[Theorem 2.5]{BG}, \cite[Theorem 8.6]{N} and 
\cite[Proposition 7.6]{CH}. It follows that ${\mathcal P}(R)$ coincides 
with the class of projective modules, if $R$ has finite weak global 
dimension. It is also known that ${\mathcal P}(R)$-modules are necessarily
projective if one of the following two conditions holds:
\newline
(a) $R$ is right coherent and $\mbox{pd}_{R^{op}}C < \infty$ for any 
finitely presented right $R$-module $C$,
\newline
(b) there is an integer $n \geq 0$, such that $\mbox{pd}_RC \leq n$ 
for any finitely presented $R$-module $C$.
\newline
The contractibility of acyclic complexes of projective modules under 
assumption (a) is proved in \cite[Theorems 2.2 and 4.5]{GI} (see also 
\cite{CET}), whereas the same conclusion under assumption (b) follows 
from \cite[Remark 2.10(i)]{ET25a}, by letting ${\tt U}$ be the class 
of projective modules therein.

In this paper, we enlarge the list of rings over which acyclic 
complexes of projective modules are necessarily contractible and
incorporate certain group algebras therein. We consider the group 
class $\overline{\mathfrak{F}}$ that is obtained from the class 
$\mathfrak{F}$ of finite groups by repeated applications of 
Kropholler's operation ${\scriptstyle{{\bf LH}}}$ \cite{Krop} and 
Talelli's operation $\Phi$ \cite{T07}; see Section 1 for details on 
this transfinite process. The class $\overline{\mathfrak{F}}$ is 
admittedly large. Its subclass ${\scriptstyle{{\bf LH}}}\mathfrak{F}$ 
is already closed under extensions, ascending unions, amalgamated 
free products and HNN extensions; it includes all finite groups, 
all free groups, all soluble groups and all groups of automorphisms 
of Noetherian modules over commutative rings; cf.\ \cite{Krop}. For 
an example of a group that lies outside $\overline{\mathfrak{F}}$, 
the reader is referred to \cite[Appendix]{ET25b}. Our main result is 
the following:

\medskip

\noindent
{\bf Theorem.}
{\em If $k$ is a commutative ring and $G \in \overline{\mathfrak{F}}$,
the following conditions are equivalent:

(i) The order of any finite subgroup of $G$ is a unit in $k$ and any 
acyclic complex of projective $k$-modules is contractible.

(ii) Any acyclic complex of projective $kG$-modules is contractible.}

\medskip

\noindent
We note that condition (i) in the Theorem above is satisfied if 
$k = \mathbb{Q}$ (and $G$ is any $\overline{\mathfrak{F}}$-group) 
or else if $k = \mathbb{Z}$ and $G$ is a torsion-free 
$\overline{\mathfrak{F}}$-group. We present some applications 
regarding modules over group algebras that admit complete 
projective resolutions (in particular, regarding groups with 
periodic cohomology after some steps) and elaborate on the 
relation between the Gorenstein and the ordinary cohomological 
dimensions, in Section 4.

Over any ring $R$, there is a close relation between the contractibility 
of acyclic complexes of projective modules, the pure acyclicity of acyclic 
complexes of flat modules and the contractibility of acyclic complexes of 
injective right $R$-modules; see, for example, \cite[Corollary 2.2]{ET25b}. 
Hence, there are versions of the theorem above, regarding acyclic complexes 
of flat or injective modules; these are detailed in Sections 2 and 3 of the 
paper. Appendices A and B deal with some technical properties of fibrant 
modules that are needed in Section 3.

\medskip

\noindent
{\em Notations and terminology.}
All rings considered in the paper are unital and associative. Unless 
otherwise specified, a module over a ring $R$ will be assumed to be 
a left $R$-module. We denote by ${\tt Proj}(R)$, ${\tt Flat}(R)$ and 
${\tt Inj}(R)$ the classes of projective, flat and injective $R$-modules.
We also denote by $\mathcal{P}(R)$, $\mathcal{F}(R)$ and $\mathcal{I}(R)$ 
the class of modules that appear as cokernels/kernels of acyclic 
complexes of projective, flat and injective $R$-modules.

\section{Preliminaries}

\noindent
In this Section, we record some preliminary results that will be used 
in the text, regarding the hierarchical structure of groups in the class 
$\overline{\mathfrak{F}}$, the notion of Benson's cofibrant modules and 
some basic properties of the cotorsion pair that is induced by these modules.

\medskip

\noindent
{\sc I.\ The group class $\overline{\mathfrak{F}}$.}
We only consider group classes that are closed under isomorphisms 
and contain the trivial group. First of all, we recall the operation 
${\scriptstyle{{\bf LH}}}$ that was introduced by Kropholler in 
\cite{Krop}. If $\mathfrak{C}$ is any group class, we define for each 
ordinal $\alpha$ the group class 
${\scriptstyle{{\bf H}}}_{\alpha}\mathfrak{C}$ using transfinite 
induction, as follows: We let
${\scriptstyle{{\bf H}}}_0\mathfrak{C} = \mathfrak{C}$. For an ordinal 
$\alpha >0$, the class ${\scriptstyle{{\bf H}}}_{\alpha}\mathfrak{C}$ 
consists of those groups $G$ which admit a cellular action on a finite 
dimensional contractible CW-complex $X$, in such a way that each isotropy
subgroup of the action belongs to 
${\scriptstyle{{\bf H}}}_{\beta}\mathfrak{C}$ for some ordinal
$\beta < \alpha$. We say that a group belongs to 
${\scriptstyle{{\bf H}}}\mathfrak{C}$ if it belongs to
${\scriptstyle{{\bf H}}}_{\alpha}\mathfrak{C}$ for some $\alpha$. The
class ${\scriptstyle{{\bf LH}}}\mathfrak{C}$ consists of those groups 
$G$, all of whose finitely generated subgroups $H$ are contained in an
${\scriptstyle{{\bf H}}}\mathfrak{C}$-subgroup $K=K(H) \subseteq G$.
If the class $\mathfrak{C}$ is subgroup-closed, then the class
${\scriptstyle{{\bf H}}}\mathfrak{C}$ is subgroup-closed as well. In 
that case, ${\scriptstyle{{\bf LH}}}\mathfrak{C}$ consists of those 
groups all of whose finitely generated subgroups are 
${\scriptstyle{{\bf H}}}\mathfrak{C}$-groups; hence, the class
${\scriptstyle{{\bf LH}}}\mathfrak{C}$ is also subgroup-closed.

We are also interested in the operation $\Phi$ that was introduced by 
Talelli \cite{T07}, for the class of finite groups; see also \cite{Bi,MS}. 
If $k$ is a commutative ring and $\mathfrak{C}$ is a group class, then 
we define $\Phi(k)\mathfrak{C}$, the result of applying the operation 
$\Phi(k)$ to $\mathfrak{C}$, as the class consisting of those groups 
$G$, for which the following two conditions are equivalent for any 
$kG$-module $M$:
\newline
(i) $\mbox{pd}_{kG}M < \infty$,
\newline
(ii) there is an integer $n$, such that
$\mbox{pd}_{kC}\mbox{res}_C^GM \leq n$ for any
$\mathfrak{C}$-subgroup $C \subseteq G$.
\newline
We shall also consider the version of the operation $\Phi(k)$ for the 
flat (resp.\ injective) dimension and define for any group class 
$\mathfrak{C}$ the class $\Phi(k)_{flat} \mathfrak{C}$ (resp.\ 
$\Phi(k)_{inj} \mathfrak{C}$) accordingly.

\begin{Lemma}
If $k$ is a commutative ring and $\mathfrak{C}$ is a subgroup-closed 
class of groups, then the classes $\Phi(k) \mathfrak{C}$, 
$\Phi(k)_{flat} \mathfrak{C}$ and $\Phi(k)_{inj} \mathfrak{C}$ are 
subgroup-closed as well.
\end{Lemma}

\vspace{-0.05in}

\noindent
{\em Proof.} We provide the argument regarding $\Phi(k) \mathfrak{C}$ 
and note that we can similarly take care of the flat and injective 
versions. Let $G$ be a $\Phi(k) \mathfrak{C}$-group and consider a 
subgroup $H \subseteq G$. In order to show that 
$H \in \Phi(k) \mathfrak{C}$, assume that $M$ is a $kH$-module and 
$n \geq 0$ is such that $\mbox{pd}_{kC}\mbox{res}_C^HM \leq n$ for 
any $\mathfrak{C}$-subgroup $C \subseteq H$. We claim that the 
induced $kG$-module $\mbox{ind}_H^GM$ has finite projective 
dimension. To verify this, it suffices to show that 
$\mbox{pd}_{kC}\mbox{res}_C^G\mbox{ind}_H^GM \leq n$ for any 
$\mathfrak{C}$-subgroup $C \subseteq G$. Since $\mathfrak{C}$ is closed 
under subgroups and isomorphisms, this follows from our assumption on 
$M$ and Mackey's double coset formula (cf.\ 
\cite[Chapter III, Proposition 5.6(ii)]{Br}). Having proved that 
$\mbox{pd}_{kG} \mbox{ind}_H^GM < \infty$, we may conclude that 
$\mbox{pd}_{kH} \mbox{res}_H^G\mbox{ind}_H^GM < \infty$ as well. Since 
$M$ is a direct summand of $\mbox{res}_H^G\mbox{ind}_H^GM$, it follows 
that $\mbox{pd}_{kH}M < \infty$. \hfill $\Box$

\medskip

\noindent
Assume that $k \longrightarrow K$ is a homomorphism of commutative rings, 
such that the multiplication map $m : K \otimes_kK \longrightarrow K$ is 
bijective.\footnote{This condition is equivalent to (a) 
$k \longrightarrow K$ is an epimorphism in the category of commutative 
rings and also to (b) $f_1 = f_2 : K \longrightarrow K \otimes_kK$, 
where $f_1(x) = x \otimes 1$ and $f_2(x) = 1 \otimes x$ for all 
$x \in K$.} In this case, for any $K$-module $M$ the canonical map
$K \otimes_kM \longrightarrow M$ is bijective as well. It follows
that for any other $K$-module $N$ we have an equality  
$\mbox{Hom}_K(M,N) = \mbox{Hom}_k(M,N)$.

\begin{Lemma}
Let $k \longrightarrow K$ be a flat homomorphism of commutative rings, 
such that the multiplication map $m : K \otimes_kK \longrightarrow K$ 
is bijective, and consider a group class $\mathfrak{C}$.

(i) If $\mbox{pd}_kK < \infty$, then 
$\Phi(k) \mathfrak{C} \subseteq \Phi(K) \mathfrak{C}$.

(ii) $\Phi(k)_{flat} \mathfrak{C} \subseteq \Phi(K)_{flat} \mathfrak{C}$

(iii) $\Phi(k)_{inj} \mathfrak{C} \subseteq \Phi(K)_{inj} \mathfrak{C}$.
\end{Lemma}

\vspace{-0.05in}

\noindent
{\em Proof.} 
(i) Let $G$ be a $\Phi(k) \mathfrak{C}$-group. In order to show that 
$G \in \Phi(K) \mathfrak{C}$, we consider a $KG$-module and a non-negative
integer $n$, such that $\mbox{pd}_{KC}\mbox{res}_C^GM \leq n$ for any 
$\mathfrak{C}$-subgroup $C \subseteq G$. Since $\mbox{pd}_kK = m < \infty$,
it follows that $\mbox{pd}_{kC}KC = m$ and hence 
$\mbox{pd}_{kC}\mbox{res}_C^GM \leq n+m$ for any $\mathfrak{C}$-subgroup 
$C \subseteq G$. But $G \in \Phi(k) \mathfrak{C}$ and hence we conclude 
that $\mbox{pd}_{kG}M < \infty$. The $k$-flatness of $K$ implies that
$\mbox{pd}_{KG}(K \otimes_kM) < \infty$, so that $\mbox{pd}_{KG}M < \infty$.

(ii) We proceed along the same lines as above. Here, the $k$-flatness of 
$K$ implies that $KG$ is flat as a $kG$-module, in view of the Lazard-Govorov 
theorem; cf.\ \cite[Theorem 4.34]{L}. Hence, any flat $KG$-module is also 
flat as a $kG$-module.

(iii) Let $G$ be a $\Phi(k)_{inj} \mathfrak{C}$-group. In order to show 
that $G \in \Phi(K)_{inj} \mathfrak{C}$, consider a $KG$-module $M$ and 
a non-negative integer $n$, such that $\mbox{id}_{KC}\mbox{res}_C^GM \leq n$ 
for any $\mathfrak{C}$-subgroup $C \subseteq G$. Since $K$ is $k$-flat, 
any injective $KC$-module is also injective as a $kC$-module and hence 
$\mbox{id}_{kC}\mbox{res}_C^GM \leq n$ for any $\mathfrak{C}$-subgroup 
$C \subseteq G$. But $G \in \Phi(k)_{inj} \mathfrak{C}$ and hence
$\mbox{id}_{kG}M < \infty$, say $\mbox{id}_{kG}M =t$. We consider an 
$KG$-injective resolution ${\bf I}$ of $M$ and let $N$ be its $t$-th 
kernel. Since any injective $KG$-module is also injective as a $kG$-module, 
${\bf I}$ is also a $kG$-injective resolution of $M$; it follows that the 
$KG$-module $N$ is injective as a $kG$-module. On the other hand, for any 
$KG$-module $L$ we have an equality 
$\mbox{Hom}_{KG}(L,N) = \mbox{Hom}_{kG}(L,N)$, so that $N$ is injective 
as a $KG$-module. It follows that $\mbox{id}_{KG}M \leq t < \infty$, as 
needed. \hfill $\Box$

\begin{Corollary}
For any group class $\mathfrak{C}$, we have
$\Phi(\mathbb{Z}) \mathfrak{C} \subseteq \Phi(\mathbb{Q}) \mathfrak{C}$,
$\Phi(\mathbb{Z})_{flat} \mathfrak{C} \subseteq 
 \Phi(\mathbb{Q})_{flat} \mathfrak{C}$
and 
$\Phi(\mathbb{Z})_{inj} \mathfrak{C} \subseteq 
 \Phi(\mathbb{Q})_{inj} \mathfrak{C}$.
\hfill $\Box$
\end{Corollary}

\noindent
Having fixed the commutative ring $k$, we define a hierarchy of group 
classes ${\mathfrak F}(k)_{\alpha}$, where $\alpha$ is any ordinal 
number, as follows: 

(i) $\mathfrak{F}(k)_0 = \mathfrak{F}$ is the class of finite groups,

(ii) ${\mathfrak F}(k)_{\alpha +1} = 
      {\scriptstyle{{\bf LH}}}{\mathfrak F}(k)_{\alpha} 
      \cup \Phi(k) {\mathfrak F}(k)_{\alpha} \cup 
      \Phi(k)_{flat} {\mathfrak F}(k)_{\alpha}$
     for any ordinal $\alpha$ and 

(iii) $\mathfrak{F}(k)_{\alpha} = 
       \bigcup_{\beta < \alpha}\mathfrak{F}(k)_{\beta}$
      for any limit ordinal $\alpha$. 
\newline
Then, $\overline{\mathfrak{F}}(k)$ is defined as the class consisting 
of those groups which are contained in ${\mathfrak F}(k)_{\alpha}$, for 
some ordinal number $\alpha$. The operations ${\scriptstyle{{\bf LH}}}$, 
$\Phi(k)$ and $\Phi(k)_{flat}$ are continuous (cf.\ \cite[Lemma 1.3 
and 1.4]{ET25b}) and hence $\overline{\mathfrak{F}}(k)$ is the smallest 
class of groups that contains the class $\mathfrak{F}$ of finite groups 
and is invariant under the operation 
${\scriptstyle{{\bf LH}}} \cup \Phi(k) \cup \Phi(k)_{flat}$; it is the 
closure of $\mathfrak{F}$ under that operation. Since the operation 
${\scriptstyle{{\bf LH}}}$ preserves subgroup-closed classes, Lemma 1.1 
implies that the class $\overline{\mathfrak{F}}(k)$ is subgroup-closed. 
We denote the class $\overline{\mathfrak{F}}(\mathbb{Z})$ simply by 
$\overline{\mathfrak{F}}$ and note that Corollary 1.3 implies that 
$\overline{\mathfrak{F}} = \overline{\mathfrak{F}}(\mathbb{Z}) \subseteq 
 \overline{\mathfrak{F}}(\mathbb{Q})$. 
 
Analogously, using the operation ${\scriptstyle{{\bf LH}}} \cup \Phi(k)_{inj}$ 
in the previous paragraph (instead of the operation 
${\scriptstyle{{\bf LH}}} \cup \Phi(k) \cup \Phi(k)_{flat}$), we define the 
class $\widehat{\mathfrak{F}}(k)$ as the smallest class of groups that contains 
$\mathfrak{F}$ and remains invariant under that operation. Since
${\scriptstyle{{\bf LH}}}$ preserves subgroup-closed classes, Lemma 1.1 implies 
that the class $\widehat{\mathfrak{F}}(k)$ is also subgroup-closed. The class
$\widehat{\mathfrak{F}}(k)$ is a sub-class of $\overline{\mathfrak{F}}(k)$, since 
the operation $\Phi(k)_{inj}$ is a sub-operation of $\Phi(k)_{flat}$; cf.\ 
\cite[Lemma 1.5]{ER2}. We denote the class $\widehat{\mathfrak{F}}(\mathbb{Z})$ 
by $\widehat{\mathfrak{F}}$ and note that Corollary 1.3 implies that 
$\widehat{\mathfrak{F}} = \widehat{\mathfrak{F}}(\mathbb{Z}) \subseteq 
 \widehat{\mathfrak{F}}(\mathbb{Q})$. 

\medskip

\noindent
{\sc II.\ The diagonal action.}
We fix a commutative ring $k$ and a group $G$. We outline below certain properties 
resulting from the Hopf algebra structure of $kG$; for more details, the 
reader is referred to \cite[Chapter III]{Br}. The diagonal action of $G$ 
endows the tensor product $M \otimes_k N$ of two $kG$-modules $M,N$ with 
the structure of a $kG$-module; we define 
$g \cdot (x \otimes y) = gx \otimes gy \in M \otimes_k N$ for any $g \in G$, 
$x \in M$ and $y \in N$. The coinvariance of $M \otimes_kN$ under this 
action is the tensor product $M \otimes_{kG}N$. Analogously, the $k$-module 
$\mbox{Hom}_k(M,N)$ admits the structure of a $kG$-module, with the group 
$G$ acting (diagonally) as follows: For any $g \in G$, 
$f \in \mbox{Hom}_k(M,N)$ and $x \in M$
we define $(g \cdot f)(x) = gf(g^{-1}x)\in N$. The $G$-invariant 
submodule of $\mbox{Hom}_k(M,N)$ coincides with the $k$-module
$\mbox{Hom}_{kG}(M,N)$. 

We note that for any three $kG$-modules $L,M,N$ the natural isomorphism 
of $k$-modules
\[ \mbox{Hom}_k(L \otimes_k M,N) \simeq
   \mbox{Hom}_k(L, \mbox{Hom}_k(M,N)) \]
is $kG$-linear, where $G$ acts diagonally on the tensor product and the 
three Hom-groups that are involved. Hence, taking $G$-invariants, we 
conclude that the above isomorphism of $kG$-modules restricts to an 
isomorphism of $k$-modules
\[ \mbox{Hom}_{kG}(L \otimes_k M,N) \simeq 
   \mbox{Hom}_{kG} (L,\mbox{Hom}_k(M,N)) . \] 
It follows that the $kG$-module $L \otimes_k M$ is projective, if $L$ 
is projective (as a $kG$-module) and $M$ is $k$-projective. It also 
follows that the $kG$-module $\mbox{Hom}_k(M,N)$ is injective, if $M$ 
is $k$-flat and $N$ is injective (as a $kG$-module). The symmetric role 
played by the $kG$-modules $L,M$ in the above isomorphism implies that 
there is another natural isomorphism 
\[ \mbox{Hom}_{kG} (L,\mbox{Hom}_k(M,N)) \simeq 
   \mbox{Hom}_{kG} (M,\mbox{Hom}_k(L,N)) . \]
We conclude that the $kG$-module $\mbox{Hom}_k(M,N)$ is injective,
if $M$ is projective (as a $kG$-module) and $N$ is $k$-injective.

\medskip

\noindent
{\sc III.\ Benson's (co-)fibrant modules.}
Let $B(G,\mathbb{Z})$ be the $\mathbb{Z}G$-module of bounded 
functions from $G$ to $\mathbb{Z}$, that was considered in 
\cite{KT}. We note that $B(G,\mathbb{Z})$ is free as an abelian 
group; in fact, $B(G,\mathbb{Z})$ is $\mathbb{Z}H$-free for any 
finite subgroup $H \subseteq G$. The constant functions define a 
$\mathbb{Z}$-split $\mathbb{Z}G$-linear embedding 
$\iota : \mathbb{Z} \longrightarrow B(G,\mathbb{Z})$. An additive 
splitting is provided by evaluating functions at the identity 
element of $G$. The $kG$-module 
$B(G,k) = B(G,\mathbb{Z}) \otimes_{\mathbb{Z}}k$ is identified 
with the $kG$-module of all functions from $G$ to $k$ whose image 
if a finite set. It is a free $kH$-module for any finite subgroup 
$H \subseteq G$. The $\mathbb{Z}$-split $\mathbb{Z}G$-linear 
embedding $\iota$ induces a $k$-split $kG$-linear embedding
$\iota \otimes 1 : k \longrightarrow B(G,k)$.

Following Benson \cite{Ben}, we say that a $kG$-module $M$ is cofibrant
if the (diagonal) $kG$-module $M \otimes_k B(G,k)$ is projective. Benson
used these modules to construct an abelian model structure in a certain 
subcategory of the category of $kG$-modules (with cofibrant modules being
the cofibrant objects of the model structure). We denote by ${\tt Cof}(kG)$ 
the class of cofibrant $kG$-modules. Since the $kG$-module $B(G,k)$ is 
$k$-free, any projective $kG$-module is cofibrant, i.e.\ 
${\tt Proj}(kG) \subseteq {\tt Cof}(kG)$. On the other hand, we have 
${\tt Cof}(kG) \subseteq \mathcal{P}(kG)$: Iterating the operation 
$M \mapsto M \otimes_k B(G,k)$, we may construct a coresolution of a 
cofibrant module by projective $kG$-modules, with all cokernels of the 
coresolution cofibrant; cf.\ \cite{CK}. We say that a $kG$-module $M$ 
is fibrant if the (diagonal) $kG$-module $\mbox{Hom}_k(B(G,k),M)$ is 
injective; cf.\ \cite[$\S $4]{ER2}. Let ${\tt Fib}(kG)$ denote the 
class of fibrant $kG$-modules. Since $B(G,k)$ is $k$-free (and hence 
$k$-flat), all injective $kG$-modules are fibrant, so that 
${\tt Inj}(kG) \subseteq {\tt Fib}(kG)$. As shown in 
\cite[Proposition 4.2(iii)]{ER2}, any fibrant module is a kernel of 
an acyclic complex of injective $kG$-modules, all of whose kernels 
are fibrant; in particular, ${\tt Fib}(kG) \subseteq \mathcal{I}(kG)$.

If $H \subseteq G$ is a subgroup, then $B(H,k)$ is a direct summand 
of $B(G,k)$ as a $kH$-module; hence, any cofibrant (resp.\ fibrant) 
$kG$-module is also cofibrant (resp.\ fibrant) as a $kH$-module. 
Letting $H=1$ be the trivial subgroup of $G$, we conclude that any 
cofibrant (resp.\ fibrant) $kG$-module is $k$-projective (resp.\ 
$k$-injective). The following simple and certainly well-known result 
shows that the converse also holds, if the group $G$ is finite. In 
particular, the cofibrant $kG$-modules over a finite group coincide
with the {\em big lattices} considered in \cite{BB}.

\begin{Lemma}
Let $k$ be a commutative ring and consider a finite group $G$.

(i) A $kG$-module $M$ is cofibrant if and only if $M$ is $k$-projective.

(ii) A $kG$-module $M$ is fibrant if and only if $M$ is $k$-injective.
\end{Lemma}

\vspace{-0.05in}

\noindent
{\em Proof.}
(i) Assume that $M$ is $k$-projective. Then, the diagonal $kG$-module 
$M \otimes_k B(G,k)$ is projective, since $B(G,k)=kG$ is projective; 
hence, $M$ is cofibrant.

(ii) Assume that $M$ is $k$-injective. Then, the diagonal $kG$-module 
$\mbox{Hom}_k(B(G,k),M)$ is injective, since $B(G,k)=kG$ is projective; 
hence, $M$ is fibrant. \hfill $\Box$

\begin{Lemma}
Let $k$ be a commutative ring and $G$ a finite group. Then, the following
conditions are equivalent:

(i) Any fibrant $kG$-module is injective, i.e.\ 
${\tt Inj}(kG) = {\tt Fib}(kG)$.

(ii) Any cofibrant $kG$-module is projective, i.e.\ 
${\tt Proj}(kG) = {\tt Cof}(kG)$.

(iii) The order of $G$ is invertible in $k$.
\end{Lemma}

\vspace{-0.05in}

\noindent
{\em Proof.} 
(i)$\rightarrow$(ii): Let $D$ be an injective cogenerator of the 
category of $k$-modules and consider a cofibrant $kG$-module $M$.
Viewing $D$ as a $kG$-module with trivial action of the group $G$, 
we also consider the $kG$-module $\mbox{Hom}_k(M,D)$ and note that 
\[ \mbox{Hom}_k(B(G,k),\mbox{Hom}_k(M,D)) \simeq 
   \mbox{Hom}_k(B(G,k) \otimes_kM,D) . \]
Since the $kG$-module $B(G,k) \otimes_kM$ is projective (and hence 
flat), it follows from Lambek's theorem \cite[Theorem 4.9]{L} that 
$\mbox{Hom}_k(M,D) \in {\tt Fib}(kG)$. Then, the $kG$-module 
$\mbox{Hom}_k(M,D)$ is injective and hence, using Lambek's theorem 
once more, we conclude that the $kG$-module $M$ is flat. We have thus 
proved that all cofibrant $kG$-modules are flat. Since any cofibrant
$kG$-module is known to be a cokernel of an acyclic complex of projective 
$kG$-modules, all of whose cokernels are cofibrant (cf.\ \cite{CK}), 
an application of Neeman's result \cite[Proposition 7.6]{N} implies 
that ${\tt Cof}(kG) \subseteq {\tt Proj}(kG)$. 

(ii)$\rightarrow$(iii): In view of Lemma 1.4(i), the trivial $kG$-module
$k$ is cofibrant and hence projective. Then, the augmentation 
$\varepsilon : kG \longrightarrow k$ splits, i.e.\ there is a $kG$-linear
map $s : k \longrightarrow kG$ with $\varepsilon \circ s = 1_k$. If 
$s(1) = a \sum_{g \in G}g$ with $a \in k$, then 
$a \cdot \! \mid \! G \! \mid \; = 1 \in k$.

(iii)$\rightarrow$(i): Let $M \in {\tt Fib}(kG)$ and consider a
$kG$-linear monomorphism $f : M \longrightarrow N$. Since $M$ is 
$k$-injective (cf.\ Lemma 1.4(ii)), $f$ admits a $k$-linear left 
inverse $t : N \longrightarrow M$. We modify $t$ by the standard 
averaging process to make it $kG$-linear: If $a \in k$ is the 
inverse of the order of $G$, then the map 
$\tau = a \sum_{g \in G} g \cdot t \in \mbox{Hom}_k(N,M)$ is a 
$kG$-linear left inverse of $f$. It follows that the $kG$-module 
$M$ is injective. \hfill $\Box$

\medskip

\noindent
{\bf Remark 1.6.}
A careful examination of the proof of the implication (i)$\rightarrow$(ii)
in Lemma 1.5 above shows that all three conditions therein are also
equivalent to the following assertion: Any $kG$-module $M$ for which 
the diagonal $kG$-module $M \otimes_k B(G.k)$ is flat (such modules 
were called cofibrant-flat in \cite[$\S $2]{ER2}) is necessarily flat.
\addtocounter{Lemma}{1}

\medskip

\noindent
{\sc IV.\ Cotorsion pairs.}
Let $R$ be a ring and {\tt S} a class of modules. A continuous 
ascending filtration of a module $M$ by modules in {\tt S} is a family 
of submodules $(M_{\alpha})_{\alpha < \gamma}$ of $M$, where $\gamma$ 
is an ordinal number, such that:

(i) $M_0=0$ and $M = \bigcup_{\alpha < \gamma} M_{\alpha}$,

(ii) the family is increasing, i.e.\ $M_{\alpha} \subseteq M_{\beta}$ 
whenever $\alpha , \beta$ are two ordinals with $\alpha < \beta < \gamma$,

(iii) $M_{\beta} = \bigcup_{\alpha < \beta} M_{\alpha}$ for any
limit ordinal $\beta < \gamma$ and

(iv) $M_{\alpha +1}/M_{\alpha} \in {\tt S}$ for any ordinal
$\alpha$ with $\alpha + 1 < \gamma$.
\newline
Let Filt-{\tt S} be the class of modules that admit a continuous 
ascending filtration by modules in {\tt S}. We say that modules 
in Filt-{\tt S} are transfinite extensions of modules in {\tt S}.

The left orthogonal class $^{\perp}{\tt S}$ of {\tt S} consists of those
modules $M$, for which ${\rm Ext}^1_R(M,S)=0$ for all $S \in {\tt S}$. The 
right orthogonal ${\tt S}^{\perp}$ of {\tt S} is the class consisting of 
those modules $N$, for which ${\rm Ext}^1_R(S,N)=0$ for all $S \in {\tt S}$. 
We note that any left orthogonal class is closed under direct summands and 
transfinite extensions; in view of Eklof's lemma \cite[Theorem 7.3.4]{EJ2}, 
we always have Filt-$\left( ^{\perp}{\tt S} \right) \! = \! ^{\perp}{\tt S}$.

Let ${\tt U},{\tt V}$ be two module classes. Then, $({\tt U},{\tt V})$ 
is a cotorsion pair (cf.\ \cite[Definition 7.1.2]{EJ2}) if 
${\tt U} = {^{\perp}{\tt V}}$ and ${\tt U} ^{\perp} = {\tt V}$. The 
cotorsion pair $({\tt U},{\tt V})$ is cogenerated by a set ${\tt S}$ of 
modules if ${\tt V} = { \tt S} ^{\perp}$, so that 
${\tt U} = \, \! ^{\perp} \! \left( {\tt S} ^{\perp} \right)$. Of course, 
we then have ${\tt S} \subseteq {\tt U}$. Since all projective modules are 
contained in ${\tt U}$, we may always assume that the regular module $R$ 
is contained in ${\tt S}$. In that case, the left-hand class {\tt U} of the 
pair consists precisely of the direct summands of Filt-{\tt S} modules; cf.\ 
\cite[Corollary 7.3.5]{EJ2}. Any cotorsion pair $({\tt U},{\tt V})$ that is 
cogenerated by a set as above is complete, in the following sense: For any 
module $M$ there are short exact sequences of modules
\[ 0 \longrightarrow V \longrightarrow U \longrightarrow M \longrightarrow 0 
   \;\;\; \mbox{and} \;\;\;
   0 \longrightarrow M \longrightarrow V' \longrightarrow U' \longrightarrow 0 
   , \]
where $U,U' \in {\tt U}$ and $V,V' \in {\tt V}$; this result is proved in 
\cite{ETr}. 

(Co-)fibrant modules over a group algebra provide examples of cotorsion 
pairs as above. More precisely, let $k$ be a commutative ring and $G$ 
a group. As shown in \cite[Theorem 2.4]{ER1}, the pair 
$({\tt Cof}(kG),{\tt Cof}(kG)^{\perp})$ is a cotorsion pair that is  
cogenerated by a set. In other words, there exists a set ${\tt S}_G$ of 
cofibrant $kG$-modules, such that 
${\tt Cof}(kG)^{\perp} = { \tt S}_G ^{\perp}$ and
${\tt Cof}(kG) = \, \! ^{\perp} \! \left( {\tt S}_G ^{\perp} \right)$. 
Assuming that $kG \in {\tt S}_G$, we conclude that any cofibrant 
$kG$-module is a direct summand of a Filt-${\tt S_G}$ module.
It is shown in \cite[Theorem 4.5]{ER2} that the pair 
$(^{\perp}{\tt Fib}(kG),{\tt Fib}(kG))$ is also a cotorsion pair that 
is cogenerated by a set. In particular, both cotorsion pairs are complete.

\section{Acyclic complexes of projective/flat modules over $\overline{\mathfrak{F}}$-groups}

\noindent
In this Section, we examine the structure of cokernels of acyclic 
complexes of projective $kG$-modules, for certain pairs $(k,G)$,
in terms of modules that are induced from finite subgroups of $G$. 
As a consequence of our analysis, we prove the Theorem stated in 
the Introduction.

Having fixed the commutative ring $k$, we are interested in groups
$G$, over which all acyclic complexes of projective $kG$-modules 
are contractible, i.e.\ over which $\mathcal{P}(kG) = {\tt Proj}(kG)$. 
Let $\mathfrak{X}(k)$ be the class of these groups. We begin with 
a simple observation:

\begin{Lemma}
For any commutative ring $k$ the class $\mathfrak{X}(k)$ is 
subgroup-closed.
\end{Lemma}

\vspace{-0.05in}

\noindent
{\em Proof.} 
Let $G$ be an $\mathfrak{X}(k)$-group and consider a subgroup 
$H \subseteq G$. In order to show that $H$ is also an 
$\mathfrak{X}(k)$-group, assume that $M \in \mathcal{P}(kH)$. Then, 
$M$ is a cokernel of an acyclic complex of projective $kH$-modules 
${\bf P}$ and hence $\mbox{ind}_H^GM$ is a cokernel of the acyclic 
complex of projective $kG$-modules $\mbox{ind}_H^G{\bf P}$. It follows 
that $\mbox{ind}_H^GM \in \mathcal{P}(kG) = {\tt Proj}(kG)$. Then, the 
restricted $kH$-module $\mbox{res}_H^G \mbox{ind}_H^GM$ is projective.
Since $M$ is a direct summand of the latter, we conclude that 
$M \in {\tt Proj}(kH)$, as needed. \hfill $\Box$

\begin{Corollary}
Let $k$ be a commutative ring and assume that $G$ a group, such that 
any acyclic complex of projective $kG$-modules is contractible. Then, 
any acyclic complex of projective $k$-modules is contractible and the 
order of any finite subgroup of $G$ is a unit in $k$.
\end{Corollary}

\vspace{-0.05in}

\noindent
{\em Proof.} 
Since $G \in \mathfrak{X}(k)$, Lemma 2.1 implies that $H \in \mathfrak{X}(k)$
for any subgroup $H \subseteq G$. Then, any acyclic complex of projective 
$kH$-modules is contractible for any subgroup $H \subseteq G$. Letting $H=1$ 
be the trivial subgroup, it follows that any acyclic complex of projective 
$k$-modules is contractible. We now let $H \subseteq G$ be a finite subgroup. 
Since ${\tt Cof}(kH) \subseteq \mathcal{P}(kH)$, we have 
${\tt Cof}(kH) \subseteq {\tt Proj}(kH)$, i.e\ ${\tt Cof}(kH) = {\tt Proj}(kH)$. 
Then, Lemma 1.5 implies that the order of $H$ is a unit in $k$, as needed. 
\hfill $\Box$

\medskip

\noindent
One method of constructing $\mathcal{P}(kG)$-modules is to use induction
from finite subgroups of $G$: If $H \subseteq G$ is a finite subgroup and 
$M$ is a $k$-projective $kH$-module (i.e.\ a cofibrant $kH$-module), then 
$M$ is a cokernel of a suitable acyclic complex of projective $kH$-modules 
${\bf P}$. Then, $\mbox{ind}_H^GM$ is a cokernel of the acyclic complex of 
projective $kG$-modules $\mbox{ind}_H^G{\bf P}$, so that 
$\mbox{ind}_H^GM \in \mathcal{P}(kG)$. Our goal is to show that, for an 
$\overline{\mathfrak{F}}(k)$-group $G$, any module in $\mathcal{P}(kG)$ is 
built up from modules of the form $\mbox{ind}_H^GM$, for various pairs 
$(H,M)$ as above, using transfinite extensions and direct summands. This 
is reminiscent of the description of modules of type FP$_{\infty}$ over 
an ${\scriptstyle{{\bf LH}}}\mathfrak{F}$-group, as the direct summands 
of modules that admit a finite filtration with quotients induced from 
modules of type FP$_{\infty}$ over finite subgroups of the group; cf.\ 
\cite[Corollary 12.4]{Ben}.

We shall use the following result, that is proved in \cite{ER1}.

\begin{Theorem}
Let $k$ be a commutative ring and $G$ an $\overline{\mathfrak{F}}(k)$-group.

(i) (cf.\ \cite[Corollary 6.4]{ER1}) A $kG$-module $M$ is contained
in ${\tt Cof}(kG)^{\perp}$ if and only if 
$\mbox{res}_H^GM \in {\tt Cof}(kH)^{\perp}$ for any finite subgroup 
$H \subseteq G$.

(ii) (cf.\ \cite[Corollary 4.8]{ER1})
If any acyclic complex of projective $k$-modules is contractible,
then ${\tt Cof}(kG) = \mathcal{P}(kG)$. \hfill $\Box$
\end{Theorem}

\noindent
It follows form Theorem 2.3(i) that the class ${\tt Cof}(kG)$ of cofibrant
modules coincides, whenever $G$ is an $\overline{\mathfrak{F}}(k)$-group, 
with the class of {\em proper} cofibrant modules introduced by Dalezios 
and G\'{o}mez in \cite{DG}. Recall from $\S $1.IV that for any subgroup 
$H \subseteq G$ there is a set ${\tt S}_H$ of cofibrant $kH$-modules, such 
that ${\tt Cof}(kH)^{\perp} = {\tt S}_H^{\perp}$. We can always assume that 
${\tt S}_H$ contains the regular module $kH$. If $M$ is any $kG$-module, then 
the Eckmann-Shapiro isomorphism
\[ \mbox{Ext}^1_{kH}(\_\!\_, \mbox{res}_H^GM) \simeq 
   \mbox{Ext}^1_{kG}(\mbox{ind}_H^G \_\!\_,M) \]
shows that $\mbox{res}_H^GM \in {\tt Cof}(kH)^{\perp} = {\tt S}_H^{\perp}$ 
if and only if $M \in \! \left( \mbox{ind}_H^G {\tt S}_H \right) \! ^{\perp}$. 
We now consider the set 
\begin{equation}
{\tt T}_G = \bigcup \{ \mbox{ind}_H^G {\tt S}_H : H \subseteq G, 
 \mbox{$H$ is finite} \} 
\end{equation}
and record the following consequence of Theorem 2.3(i); see also 
\cite[Proposition 3.9]{DG}.

\begin{Corollary}
Let $k$ be a commutative ring and $G$ an $\overline{\mathfrak{F}}(k)$-group.

(i)  ${\tt Cof}(kG)^{\perp} = {\tt T}_G^{\perp}$, i.e.\ the cotorsion pair
$({\tt Cof}(kG),{\tt Cof}(kG)^{\perp})$ is cogenerated by ${\tt T}_G$.

(ii) The class ${\tt Cof}(kG)$ of cofibrant $kG$-modules coincides with the 
class of direct summands of Filt-${\tt T}_G$ modules.
\end{Corollary}

\vspace{-0.05in}

\noindent
{\em Proof.} 
Assertion (i) follows from Theorem 2.3(i), in view of the discussion 
above. Then, as explained in $\S $1.IV, assertion (ii) follows from 
\cite[Corollary 7.3.5]{EJ2} \hfill $\Box$

\begin{Proposition}
Let $k$ be a commutative ring and assume that all acyclic complexes of 
projective $k$-modules are contractible. If $G$ is an 
$\overline{\mathfrak{F}}(k)$-group, then the class $\mathcal{P}(kG)$ 
coincides with the class of direct summands of Filt-${\tt T}_G$ modules.
\end{Proposition}

\vspace{-0.05in}

\noindent
{\em Proof.} 
This follows combining Theorem 2.3(ii) with Corollary 2.4(ii). \hfill $\Box$

\medskip

\noindent
Recall that, if $R$ is any ring, we denote by $\mathcal{F}(R)$ the class 
of those $R$-modules that appear as cokernels of acyclic complexes of flat 
modules. Flat modules are contained in $\mathcal{F}(R)$, so that we always 
have ${\tt Flat}(R) \subseteq \mathcal{F}(R)$. It is clear that the cokernels 
of an acyclic complex of flat modules ${\bf F}$ are flat if and only if
${\bf F}$ is pure-acyclic, i.e\ if and only if the complex of abelian 
groups $L \otimes _R {\bf F}$ is acyclic for all right $R$-modules $L$.
It follows that $\mathcal{F}(R) = {\tt Flat}(R)$ if and only if any acyclic
complex of flat modules is pure-acyclic.

\begin{Theorem}
If $k$ is a commutative ring and $G$ an $\overline{\mathfrak{F}}(k)$-group,
then the following conditions are equivalent:

(i) The order of any finite subgroup of $G$ is a unit in $k$ and any 
acyclic complex of projective $k$-modules is contractible.

(ii) Any acyclic complex of projective $kG$-modules is contractible,
i.e.\ $\mathcal{P}(kG) = {\tt Proj}(kG)$.

(iii) Any acyclic complex of flat $kG$-modules is pure-acyclic, i.e.\ 
$\mathcal{F}(kG) = {\tt Flat}(kG)$.
\end{Theorem}

\vspace{-0.05in}

\noindent
{\em Proof.} 
The implication (ii)$\rightarrow$(i) follows from Corollary 2.2 (which
holds without any assumption on the group $G$). On the other hand, the 
equivalence (ii)$\leftrightarrow$(iii) is known to hold over any ring; 
see, for example, \cite[Corollary 2.2]{ET25b}. Hence it only remains to
show that (i)$\rightarrow$(ii).

Assume that assertion (i) holds and let $H \subseteq G$ be a finite 
subgroup. Since the order of $H$ is invertible in $k$, Lemma 1.5 
implies that ${\tt Cof}(kH) = {\tt Proj}(kH)$ and hence 
$\mbox{ind}_H^G{\tt Cof}(kH) \subseteq {\tt Proj}(kG)$. Since this is 
the case for any finite subgroup $H \subseteq G$, we conclude that the
set ${\tt T}_G$ in (1) above consists of projective $kG$-modules. On 
the other hand, any direct summand of a transfinite extension of 
projective $kG$-modules is a projective $kG$-module as well. In view 
of our assumption that any acyclic complex of projective $k$-modules 
is contractible, Proposition 2.5 implies that 
$\mathcal{P}(kG) \subseteq {\tt Proj}(kG)$,
so that $\mathcal{P}(kG) = {\tt Proj}(kG)$. \hfill $\Box$

\medskip

\noindent
Recall from the Introduction (and the references therein) that any
acyclic complex of projective $k$-modules is contractible, if the 
commutative ring $k$ satisfies one of the following conditions:

(a) $k$ is coherent and $\mbox{pd}_kC < \infty$ for any finitely 
presented $k$-module $C$,

(b) there is an integer $n \geq 0$, such that $\mbox{pd}_kC \leq n$ 
for any finitely presented $k$-module $C$,

(c) $k$ has finite weak global dimension.

\begin{Corollary}
Let $G$ be an $\overline{\mathfrak{F}}(\mathbb{Q})$-group.

(i) Any acyclic complex of projective $\mathbb{Q}G$-modules is 
contractible, i.e.\ $\mathcal{P}(\mathbb{Q}G) = {\tt Proj}(\mathbb{Q}G)$.

(ii) Any acyclic complex of flat $\mathbb{Q}G$-modules is pure-acyclic, 
i.e.\ $\mathcal{F}(\mathbb{Q}G) = {\tt Flat}(\mathbb{Q}G)$.  \hfill $\Box$
\end{Corollary}

\noindent
Recall that, for notational simplicity, we denote the class
$\overline{\mathfrak{F}}(\mathbb{Z})$ by $\overline{\mathfrak{F}}$.

\begin{Corollary}
Let $G$ be a torsion-free $\overline{\mathfrak{F}}$-group.

(i) Any acyclic complex of projective $\mathbb{Z}G$-modules is 
contractible, i.e.\ $\mathcal{P}(\mathbb{Z}G) = {\tt Proj}(\mathbb{Z}G)$.

(ii) Any acyclic complex of flat $\mathbb{Z}G$-modules is pure-acyclic, 
i.e.\ $\mathcal{F}(\mathbb{Z}G) = {\tt Flat}(\mathbb{Z}G)$.  \hfill $\Box$
\end{Corollary}

\section{Acyclic complexes of injective modules over $\widehat{\mathfrak{F}}$-groups}

\noindent
In this Section, we examine the kernels of acyclic  complexes of 
injective $kG$-modules, for certain pairs $(k,G)$. The discussion 
below parallels that in Section 2. 

We are interested in groups $G$, over which all acyclic complexes 
of injective $kG$-modules are contractible, i.e.\ over which 
$\mathcal{I}(kG) = {\tt Inj}(kG)$. Let $\mathfrak{Y}(k)$ be the 
class of these groups.

\begin{Lemma}
For any commutative ring $k$ the class $\mathfrak{Y}(k)$ is 
subgroup-closed.
\end{Lemma}

\vspace{-0.05in}

\noindent
{\em Proof.} 
Let $G$ be a $\mathfrak{Y}(k)$-group and consider a subgroup 
$H \subseteq G$. In order to show that $H$ is also a 
$\mathfrak{Y}(k)$-group, assume that $M \in \mathcal{I}(kH)$. Then, 
$M$ is a kernel of an acyclic complex of injective $kH$-modules 
${\bf I}$ and hence $\mbox{coind}_H^GM$ is a kernel of the acyclic 
complex of injective $kG$-modules $\mbox{coind}_H^G{\bf I}$. It 
follows that $\mbox{coind}_H^GM \in \mathcal{I}(kG) = {\tt Inj}(kG)$. 
Then, the restricted $kH$-module $\mbox{res}_H^G \mbox{coind}_H^GM$ 
is injective. Since $M$ is a direct summand of the latter, we conclude 
that $M \in {\tt Inj}(kH)$, as needed. \hfill $\Box$

\begin{Corollary}
Let $k$ be a commutative ring and assume that $G$ a group, such that 
any acyclic complex of injective $kG$-modules is contractible. Then, 
any acyclic complex of injective $k$-modules is contractible and the 
order of any finite subgroup of $G$ is a unit in $k$.
\end{Corollary}

\vspace{-0.05in}

\noindent
{\em Proof.} 
Since $G \in \mathfrak{Y}(k)$, Lemma 3.1 implies that $H \in \mathfrak{Y}(k)$
for any subgroup $H \subseteq G$. Then, any acyclic complex of injective 
$kH$-modules is contractible for any subgroup $H \subseteq G$. Letting 
$H=1$, it follows that any acyclic complex of injective $k$-modules is 
contractible. We now let $H \subseteq G$ be a finite subgroup. Since 
${\tt Fib}(kH) \subseteq \mathcal{I}(kH)$, we have 
${\tt Fib}(kH) \subseteq {\tt Inj}(kH)$, i.e\ ${\tt Fib}(kH) = {\tt Inj}(kH)$. 
Then, Lemma 1.5 implies that the order of $H$ is a unit in $k$. \hfill $\Box$

\medskip

\noindent
We shall use the following result. The proof of assertion (i) is somehow
technical and will be given in Appendix B at the end of the paper, whereas 
assertion (ii) is proved in \cite{ER2}.

\begin{Theorem}
Let $k$ be a commutative ring and $G$ an $\widehat{\mathfrak{F}}(k)$-group.

(i) A $kG$-module $M$ is contained in $^{\perp}{\tt Fib}(kG)$ if and only 
if $\mbox{res}_H^GM \in \!\!\, ^{\perp}{\tt Fib}(kH)$ for any finite subgroup 
$H \subseteq G$.

(ii) (cf.\ \cite[Proposition 4.3]{ER2})
If any acyclic complex of injective $k$-modules is contractible,
then ${\tt Fib}(kG) = \mathcal{I}(kG)$. \hfill $\Box$
\end{Theorem}

\noindent
We can now state and prove the analogue of Theorem 2.6 for acyclic
complexes of injectives. 

\begin{Theorem}
If $k$ is a commutative ring and $G$ an $\widehat{\mathfrak{F}}(k)$-group,
then the following conditions are equivalent:

(i) The order of any finite subgroup of $G$ is a unit in $k$ and any 
acyclic complex of injective $k$-modules is contractible.

(ii) Any acyclic complex of injective $kG$-modules is contractible,
i.e.\ $\mathcal{I}(kG) = {\tt Inj}(kG)$.
\end{Theorem}

\vspace{-0.05in}

\noindent
{\em Proof.} 
The implication (ii)$\rightarrow$(i) follows from Corollary 3.2 (which
holds without any assumption on the group $G$). We now assume that (i) 
holds and let $H \subseteq G$ be a finite subgroup. Since the order of 
$H$ is invertible in $k$, Lemma 1.5 implies that 
${\tt Fib}(kH) = {\tt Inj}(kH)$ and hence the left orthogonal class
$^{\perp}{\tt Fib}(kH)$ is the class of all $kH$-modules. Invoking 
Theorem 3.3(i), we conclude that the left orthogonal class
$^{\perp}{\tt Fib}(kG)$ is the class of all $kG$-modules. It follows 
that ${\tt Fib}(kG) = \left(  ^{\perp}{\tt Fib}(kG) \right) \! ^{\perp}$
is the class of injective $kG$-modules. In view of our assumption that 
any acyclic complex of injective $k$-modules is contractible, Theorem 
3.3(ii) implies that $\mathcal{I}(kG) = {\tt Fib}(kG)$ and hence
$\mathcal{I}(kG) = {\tt Inj}(kG)$. \hfill $\Box$

\begin{Corollary}
Let $G$ be an $\widehat{\mathfrak{F}}(\mathbb{Q})$-group. Then, any 
acyclic complex of injective $\mathbb{Q}G$-modules is contractible, 
i.e.\ $\mathcal{I}(\mathbb{Q}G) = {\tt Inj}(\mathbb{Q}G)$. \hfill $\Box$
\end{Corollary}

\noindent
Recall that, for notational simplicity, we denote the class
$\widehat{\mathfrak{F}}(\mathbb{Z})$ by $\widehat{\mathfrak{F}}$.

\begin{Corollary}
Let $G$ be a torsion-free $\widehat{\mathfrak{F}}$-group. Then, any 
acyclic complex of injective $\mathbb{Z}G$-modules is contractible, 
i.e.\ $\mathcal{I}(\mathbb{Z}G) = {\tt Inj}(\mathbb{Z}G)$. \hfill $\Box$
\end{Corollary}

\noindent
{\bf Remark 3.7.}
As we have noted in $\S $1.I, the class $\widehat{\mathfrak{F}}(k)$
is a subclass of $\overline{\mathfrak{F}}(k)$ for any commutative 
ring $k$. Hence, we may invoke Theorem 2.6 and conclude that the 
equivalent conditions of Theorem 3.4 are also equivalent to the 
contractibility of any acyclic complex of projective $kG$-modules 
(i.e.\ to the equality $\mathcal{P}(kG) = {\tt Proj}(kG)$) and also 
to the pure-acyclicity of any acyclic complex of flat $kG$-modules 
(i.e.\ to the equality $\mathcal{F}(kG) = {\tt Flat}(kG)$). In analogy,
the same conclusions also apply to Corollaries 3.5 and 3.6.  
\addtocounter{Lemma}{1}

\section{Applications}

\noindent
We present a few applications of the results obtained earlier,
concerning modules which admit complete resolutions, groups with 
periodic cohomology after some steps and the relation between the
Gorenstein and the classical homological dimensions for modules 
over certain group algebras. 

We consider the group classes 
$\overline{\mathfrak{F}} = \overline{\mathfrak{F}}(\mathbb{Z})$ and 
$\widehat{\mathfrak{F}} = \widehat{\mathfrak{F}}(\mathbb{Z})$. As we
have observed in $\S $1.I, there are inclusions 
$\overline{\mathfrak{F}} \subseteq \overline{\mathfrak{F}}(\mathbb{Q})$ 
and $\widehat{\mathfrak{F}} \subseteq \widehat{\mathfrak{F}}(\mathbb{Q})$. 
These four group classes are pictured in the following diagram
\[
\begin{array}{ccc}
 \overline{\mathfrak{F}} & \longrightarrow & 
 \overline{\mathfrak{F}}(\mathbb{Q}) \\
 \uparrow & & \uparrow \\
 \widehat{\mathfrak{F}} & \longrightarrow & 
 \widehat{\mathfrak{F}}(\mathbb{Q})
\end{array} \]
where all arrows are inclusions. We note that
${\scriptstyle{{\bf LH}}}\mathfrak{F}$ is contained in 
$\widehat{\mathfrak{F}}$ (and hence in all of these four 
group classes).

\medskip

\noindent
{\sc I. Complete resolutions.}
Let $k$ be a commutative ring and $G$ a group. We say that a $kG$-module 
$M$ admits a complete projective resolution of coincidence index $n$ if 
there exists a projective resolution of $M$ whose $n$-th syzygy 
$\Omega^nM$ is contained in $\mathcal{P}(kG)$. Of course, any 
$kG$-module $M$ with $\mbox{pd}_{kG}M \leq n$ admits a complete projective 
resolution of coincidence index $n$. We say that a $kG$-module admits a 
complete projective resolution if it admits a complete projective resolution 
of coincidence index $n$, for some non-negative integer $n$.

\begin{Proposition}
Let $n$ be a non-negative integer.

(i) If $G \in \overline{\mathfrak{F}}(\mathbb{Q})$, then a 
$\mathbb{Q}G$-module $M$ admits a complete projective resolution of 
coincidence index $n$ if and only if $\mbox{pd}_{\mathbb{Q}G}M \leq n$.

(ii) If $G \in \overline{\mathfrak{F}}$ is torsion-free, then a 
$\mathbb{Z}G$-module $M$ admits a complete projective resolution of 
coincidence index $n$ if and only if $\mbox{pd}_{\mathbb{Z}G}M \leq n$.
\end{Proposition}

\vspace{-0.05in}

\noindent
{\em Proof.}
(i) If $M$ admits a complete projective resolution of coincidence index 
$n$, then there is a projective resolution of $M$ whose $n$-th syzygy 
$\Omega^nM$ is contained in $\mathcal{P}(\mathbb{Q}G)$. Since 
$\mathcal{P}(\mathbb{Q}G) = {\tt Proj}(\mathbb{Q}G)$ (cf.\ Corollary 
2.7(i)), it follows that $\Omega^nM \in {\tt Proj}(\mathbb{Q}G)$ and 
hence $\mbox{pd}_{\mathbb{Q}G}M \leq n$.

(ii) This follows as in (i) above, using Corollary 2.8(i). \hfill $\Box$

\medskip

\noindent
We say that the group $G$ admits a complete projective resolution (of 
coincidence index $n$) if the trivial $\mathbb{Z}G$-module $\mathbb{Z}$ 
admits a complete projective resolution (of coincidence index $n$). In 
that case, any subgroup $H \subseteq G$ admits a complete projective 
resolution (of coincidence index $n$) as well and, for any commutative 
coefficient ring $k$, the trivial $kG$-module $k$ admits a complete 
projective resolution (of coincidence index $n$).\footnote{Indeed,
any acyclic complex of projective $\mathbb{Z}G$-modules ${\bf P}$ 
is necessarily contractible as a complex of abelian groups and hence 
$k \otimes_{\mathbb{Z}} {\bf P}$ is an {\em acyclic} complex of 
projective $kG$-modules, even if $k$ is not $\mathbb{Z}$-flat.} 

\begin{Corollary}
Let $n$ be a non-negative integer and $G$ a group that admits a complete 
projective resolution of coincidence index $n$.

(i) If $H \subseteq G$ is an
$\overline{\mathfrak{F}}(\mathbb{Q})$-subgroup, then 
$\mbox{cd}_{\mathbb{Q}}H \leq n$.

(ii) If $H \subseteq G$ is a torsion-free 
$\overline{\mathfrak{F}}$-subgroup, then 
$\mbox{cd}_{\mathbb{Z}}H \leq n$. \hfill $\Box$
\end{Corollary}

\noindent
A special class of groups that admit complete resolutions is the class
of groups with periodic cohomology after some steps. Following \cite{T80},
we say that $G$ has cohomological period $q$ after $s$ steps (where $q>0$ 
and $s \geq 0$) if there is a projective resolution of the trivial 
$\mathbb{Z}G$-module $\mathbb{Z}$ with syzygies $(R_n)_n$, such that 
$R_s \simeq R_{q+s}$. In that case, $R_s \in \mathcal{P}(\mathbb{Z}G)$ 
and the group cohomology functors $H^n(G,\_\!\_)$ and $H^{n+q}(G,\_\!\_)$ 
are naturally isomorphic if $n>s$; in particular, $G$ admits a complete 
projective resolution of coincidence index $s$. If $s \geq 0$, we say 
that $G$ has periodic cohomology after $s$ steps if $G$ has period $q$ 
after $s$ steps, for some $q>0$.

The problem of deciding which groups $G$ admit a finite dimensional 
free $G$-CW-complex homotopy equivalent to a sphere seems to be 
difficult. Such a group has necessarily periodic cohomology after some 
steps; this follows from \cite[Proposition 5.1 and Corollary 5.2]{MT}.
It is proved in \cite[Theorem B]{MT} that the converse of the latter 
assertion is true, in the case where $G$ is an 
${\scriptstyle{{\bf H}}}\mathfrak{F}$-group that has a bound on the 
orders of its finite subgroups. It is conjectured in \cite{MT} that 
a torsion-free group $G$ has periodic cohomology after some steps 
only if $G$ has finite cohomological dimension (so that 
$G \in {\scriptstyle{{\bf H}}}_1\mathfrak{F}$). We prove this 
conjecture for torsion-free groups in $\overline{\mathfrak{F}}$ and 
obtain thereby a plethora of groups $G$ which do not admit a finite 
dimensional free $G$-CW-complex homotopy equivalent to a sphere. 

The following result was proved in \cite[Theorem 4.9]{MT} for 
${\scriptstyle{{\bf H}}}\mathfrak{F}$-groups, with a slightly 
weaker bound for the rational cohomological dimension, although 
the proofs there are different. Here, our approach is based on the 
properties of complete projective resolutions.

\begin{Corollary}
Let $G$ be a group with periodic cohomology after $s$ steps, for 
some $s \geq 0$.

(i) If $H \subseteq G$ is an
$\overline{\mathfrak{F}}(\mathbb{Q})$-subgroup, then 
$\mbox{cd}_{\mathbb{Q}}H \leq s$.

(ii) If $H \subseteq G$ is a torsion-free 
$\overline{\mathfrak{F}}$-subgroup, then 
$\mbox{cd}_{\mathbb{Z}}H \leq s$. \hfill $\Box$
\end{Corollary}

\begin{Corollary}
Let $G$ be a group containing a torsion-free subgroup
$H \in \overline{\mathfrak{F}} \setminus 
 {\scriptstyle{{\bf H}}}_1\mathfrak{F}$.
Then, $G$ does not have periodic cohomology after some steps. 
In particular, $G$ does not admit a finite dimensional free 
$G$-CW-complex homotopy equivalent to a sphere. \hfill $\Box$
\end{Corollary}

\noindent
{\bf Remarks 4.5.}
(i) There are many torsion-free groups contained in 
$\overline{\mathfrak{F}} \setminus 
 {\scriptstyle{{\bf H}}}_1\mathfrak{F}$.
In fact, for any countable ordinal number $\alpha \geq 1$ there 
exists a (finitely generated simple) torsion-free group in 
${\scriptstyle{{\bf H}}}_{\alpha +1} \mathfrak{F} 
 \setminus {\scriptstyle{{\bf H}}}_{\alpha}\mathfrak{F}$;
see \cite[Theorem 5.16]{FS}.
 
(ii) Let $k$ be a commutative ring and $G$ an 
$\overline{\mathfrak{F}}(k)$-group. (Recall that 
the class $\overline{\mathfrak{F}}(k)$ contains all 
${\scriptstyle{{\bf LH}}}\mathfrak{F}$-groups.) Assume that 
any acyclic complex of projective $k$-modules is contractible 
and all finite subgroups of $G$ have order that is invertible 
in $k$. Invoking Theorem 2.6, the results above admit 
generalizations in this setting. In particular, the following 
hold:

(iia) If $n$ is a non-negative integer and $M$ is a $kG$-module that 
admits a complete projective resolution of coincidence index $n$, 
then $\mbox{pd}_{kG}M \leq n$.

(iib) If $n$ is a non-negative integer and $G$ admits a complete projective 
resolution of coincidence index $n$, then $\mbox{cd}_kG \leq n$.

(iic) If $s \geq 0$ and $G$ has periodic cohomology after $s$ steps, then 
$\mbox{cd}_kG \leq s$.
\addtocounter{Lemma}{1}

\medskip

\noindent
{\sc II. Gorenstein homological dimensions.}
Gorenstein homological algebra is the relative homological theory, that 
is based upon the classes of Gorenstein projective, Gorenstein flat and 
Gorenstein injective modules, which were defined in \cite{EJ1,EJT}; see 
also \cite{Hol}.

Even though we are only interested in group algebras here, let $R$ 
be any ring. An acyclic complex of projective modules is called 
totally acyclic if it remains acyclic after applying the functor 
$\mbox{Hom}_R(\_\!\_,P)$ for any projective module $P$. A module 
is Gorenstein projective if it is a cokernel of a totally acyclic 
complex of projective modules; let ${\tt GProj}(R)$ be the class 
of these modules. Then, it is clear that 
${\tt Proj}(R) \subseteq {\tt GProj}(R) \subseteq \mathcal{P}(R)$.
An acyclic complex of flat modules is called totally acyclic if it 
remains acyclic after applying the functor $I \otimes_R\_\!\_$ for 
any injective right module $I$. A module is Gorenstein flat if it 
is a cokernel of a totally acyclic complex of flat modules. It is 
also clear that 
${\tt Flat}(R) \subseteq {\tt GFlat}(R) \subseteq \mathcal{F}(R)$. 
Finally, an acyclic complex of injective modules is called totally 
acyclic if it remains acyclic after applying the functor 
$\mbox{Hom}_R(I,\_\!\_)$ for any injective module $I$. A module is 
Gorenstein injective if it is a kernel of a totally acyclic complex 
of injective modules; let ${\tt GInj}(R)$ be the class of these 
modules. There are inclusions
${\tt Inj}(R) \subseteq {\tt GInj}(R) \subseteq \mathcal{I}(R)$.

The Gorenstein projective dimension $\mbox{Gpd}_RM$ of a module $M$ is 
the length of a shortest resolution of $M$ by Gorenstein projective 
modules. We have an inequality $\mbox{Gpd}_RM \leq \mbox{pd}_RM$, which 
is actually an equality if $\mbox{pd}_RM < \infty$. Analogously, we may 
define the Gorenstein flat dimension $\mbox{Gfd}_RM$ (resp.\ the Gorenstein 
injective dimension $\mbox{Gid}_RM$) of $M$, using resolutions by Gorenstein 
flat modules (resp.\ coresolutions by Gorenstein injective modules). We have 
an inequality $\mbox{Gfd}_RM \leq \mbox{fd}_RM$ (resp.\ 
$\mbox{Gid}_RM \leq \mbox{id}_RM$), which is an equality if 
$\mbox{fd}_RM < \infty$ (resp.\ if $\mbox{id}_RM < \infty$). In the 
case where $R=kG$ is the algebra of a group $G$ with coefficients in a 
commutative ring $k$, the Gorenstein projective dimension (resp.\ the 
Gorenstein flat dimension) of the trivial module $k$ is the Gorenstein 
cohomogical dimension $\mbox{Gcd}_kG$ (resp.\ the Gorenstein homological
dimension $\mbox{Ghd}_kG$) of $G$ over $k$.

The following result is an immediate consequence of Corollary 2.7.

\begin{Proposition}
Let $G$ be an $\overline{\mathfrak{F}}(\mathbb{Q})$-group.

(i) Any Gorenstein projective $\mathbb{Q}G$-module is projective, 
i.e.\ ${\tt GProj}(\mathbb{Q}G) = {\tt Proj}(\mathbb{Q}G)$. For 
any $\mathbb{Q}G$-module $M$ there is an equality 
$\mbox{Gpd}_{\mathbb{Q}G}M = \mbox{pd}_{\mathbb{Q}G}M$; in particular, 
$\mbox{Gcd}_{\mathbb{Q}}G = \mbox{cd}_{\mathbb{Q}}G$.

(ii) Any Gorenstein flat $\mathbb{Q}G$-module is flat, i.e.\ 
${\tt GFlat}(\mathbb{Q}G) = {\tt Flat}(\mathbb{Q}G)$. For any 
$\mathbb{Q}G$-module $M$ there is an equality 
$\mbox{Gfd}_{\mathbb{Q}G}M = \mbox{fd}_{\mathbb{Q}G}M$; in particular, 
$\mbox{Ghd}_{\mathbb{Q}}G = \mbox{hd}_{\mathbb{Q}}G$. \hfill $\Box$
\end{Proposition}

\noindent
Analogously, the following result follows from Corollary 2.8.

\begin{Proposition}
Let $G$ be a torsion-free $\overline{\mathfrak{F}}$-group.

(i) Any Gorenstein projective $\mathbb{Z}G$-module is projective, 
i.e.\ ${\tt GProj}(\mathbb{Z}G) = {\tt Proj}(\mathbb{Z}G)$. For 
any $\mathbb{Z}G$-module $M$ there is an equality 
$\mbox{Gpd}_{\mathbb{Z}G}M = \mbox{pd}_{\mathbb{Z}G}M$; in particular, 
$\mbox{Gcd}_{\mathbb{Z}}G = \mbox{cd}_{\mathbb{Z}}G$.

(ii) Any Gorenstein flat $\mathbb{Z}G$-module is flat, i.e.\ 
${\tt GFlat}(\mathbb{Z}G) = {\tt Flat}(\mathbb{Z}G)$. For any 
$\mathbb{Z}G$-module $M$ there is an equality 
$\mbox{Gfd}_{\mathbb{Z}G}M = \mbox{fd}_{\mathbb{Z}G}M$; in particular, 
$\mbox{Ghd}_{\mathbb{Z}}G = \mbox{hd}_{\mathbb{Z}}G$. \hfill $\Box$
\end{Proposition}

\noindent
We note that the equality between the Gorenstein and the ordinary 
cohomological dimensions in Proposition 4.6(i) and Proposition 4.7(i)
has been obtained for ${\scriptstyle{{\bf LH}}}\mathfrak{F}$-groups 
in \cite[Theorem 3.5]{T14}. We conclude by stating the following 
consequence of Corollaries 3.5 and 3.6.

\begin{Proposition}
(i) Let $G$ be an $\widehat{\mathfrak{F}}(\mathbb{Q})$-group. Then,
any Gorenstein injective $\mathbb{Q}G$-module is injective, i.e.\ 
${\tt GInj}(\mathbb{Q}G) = {\tt Inj}(\mathbb{Q}G)$. For 
any $\mathbb{Q}G$-module $M$ we have 
$\mbox{Gid}_{\mathbb{Q}G}M = \mbox{id}_{\mathbb{Q}G}M$.

(ii) Let $G$ be a torsion-free $\widehat{\mathfrak{F}}$-group. 
Then, any Gorenstein injective $\mathbb{Z}G$-module is injective, 
i.e.\ ${\tt GInj}(\mathbb{Z}G) = {\tt Inj}(\mathbb{Z}G)$. For 
any $\mathbb{Z}G$-module $M$ we have  
$\mbox{Gid}_{\mathbb{Z}G}M = \mbox{id}_{\mathbb{Z}G}M$. \hfill $\Box$
\end{Proposition}

\noindent
{\bf Remarks 4.9.}
(i) Let $k$ be a commutative ring and $G$ a group, such that all finite 
subgroups of $G$ have order invertible in $k$. Invoking Theorem 2.6, if 
$G$ is an $\overline{\mathfrak{F}}(k)$-group and any acyclic complex of 
projective $k$-modules is contractible (resp.\ Theorem 3.4, if $G$ is an 
$\widehat{\mathfrak{F}}(k)$-group and any acyclic complex of injective 
$k$-modules is contractible), we can obtain generalizations of the 
results presented above, regarding Gorenstein projective/flat (resp.\ 
Gorenstein injective) modules and the associated dimensions in this 
setting.

(ii) Let $k$ be a commutative ring, such that any acyclic complex 
of projective $k$-modules is contractible, and consider an 
$\overline{\mathfrak{F}}(k)$-group $G$. Since all cofibrant $kG$-modules
are Gorenstein projective (cf.\ \cite{CK}), there are inclusions 
${\tt Cof}(kG) \subseteq {\tt GProj}(kG) \subseteq \mathcal{P}(kG)$.
Then, \cite[Corollary 4.8]{ER1} implies that  
${\tt GProj}(kG) = \mathcal{P}(kG)$. Invoking Proposition 2.5, it follows 
that any Gorenstein projective $kG$-module is a direct summand of a 
transfinite extension of $kG$-modules of the form $\mbox{ind}_H^GM$, 
where $H \subseteq G$ is a finite subgroups and $M$ is a $k$-projective
$kH$-module.
\addtocounter{Lemma}{1}

\appendix

\section{Fibrant modules and coinduction}

\noindent
In this Appendix, we prepare the ground for the proof of Theorem 
3.3(i) and study the behaviour of fibrant modules with respect to 
coinduction from subgroups. Having fixed the commutative ring $k$, 
we proceed as in \cite[$\S $5]{ER1}. For notational simplicity, in 
both Appendices A and B, we denote the $kG$-module $B(G,k)$ that is 
associated with a group $G$ by $B(G)$ and the group class operation 
$\Phi(k)_{inj}$ by $\Phi_{inj}$.

The restriction functor from a group to a subgroup always maps 
fibrant modules to fibrant modules. Indeed, for any subgroup $H$ 
of a group $G$ the $kH$-module $B(H)$ is a direct summand of 
$\mbox{res}_H^GB(G)$. In particular, all fibrant $kG$-modules 
are necessarily  $k$-injective. Let $\mathcal{I}_0(kG)$ be the 
class of those $kG$-modules that appear as kernels of acyclic 
complexes of injective $kG$-modules, that are contractible as 
complexes of $k$-modules. Since any fibrant module is a kernel 
of an acyclic complex of injective modules, all of whose kernels 
are fibrant \cite[Proposition 4.2(iii)]{ER2}, we conclude that 
${\tt Fib}(kG) \subseteq \mathcal{I}_0(kG)$. We consider the group 
class $\mathfrak{J} = \mathfrak{J}(k)$, which consists of those 
groups $H$ that have the following property: {\em If $G$ is a group 
containing $H$ as a subgroup and $M \in \mathcal{I}_0(kH)$, then the 
diagonal $kH$-module $\mbox{Hom}_k(\mbox{res}_H^GB(G),M)$ is injective.}
In order to describe the relevance of $\mathfrak{J}$-groups in the 
study of the behaviour of fibrant modules under coinduction, we need 
the following result. It is the precise Hom-coinduction version of 
\cite[Lemma 1.1]{Sw}, that deals with diagonal tensor products and 
induction.

\begin{Lemma}
Let $G$ be a group, $H \subseteq G$ a subgroup and consider a pair 
$(M,N)$, where $M$ is a $kG$-module and $N$ is a $kH$-module. Then, 
there is an isomorphism of $kG$-modules between the diagonal $kG$-module 
$\mbox{Hom}_k \! \left( M,\mbox{coind}_H^{\, G}N \right)$ and the $kG$-module 
$\mbox{coind}_H^{\, G} \mbox{Hom}_k \! \left( \mbox{res}_H^GM,N \right)$, 
which is coinduced from the diagonal $kH$-module 
$\mbox{Hom}_k \! \left( \mbox{res}_H^GM,N \right)$.
\end{Lemma}

\vspace{-0.05in}

\noindent
{\em Proof.} For any $k$-linear map 
$f : M \longrightarrow \mbox{coind}_H^{\, G}N$ we consider the associated
$k$-linear map 
$\phi : kG \longrightarrow \mbox{Hom}_k \! \left( \mbox{res}_H^GM,N \right)$,
which is defined by letting $\phi(g)(m) = f(g^{-1}m)(g) \in N$ for all 
$g \in G$ and $m \in M$. A routine check shows that (a) $\phi$ is $kH$-linear
and hence defines an element of the coinduced module 
$\mbox{coind}_H^{\, G} \mbox{Hom}_k \! \left( \mbox{res}_H^GM,N \right)$,
(b) the map $f \mapsto \phi$ is $kG$-linear for the diagonal $kG$-module 
structure on $\mbox{Hom}_k \! \left( M,\mbox{coind}_H^{\, G}N \right)$
and the coinduced $kG$-module structure on 
$\mbox{coind}_H^{\, G} \mbox{Hom}_k \! \left( \mbox{res}_H^GM,N \right)$
and (c) the map $f \mapsto \phi$ is bijective. \hfill $\Box$

\begin{Proposition}
Let $H$ be a $\mathfrak{J}$-subgroup of a group $G$. Then:

(i) the coinduction functor $\mbox{coind}_H^{\, G}$ maps 
${\tt Fib}(kH)$ into ${\tt Fib}(kG)$ and 

(ii) the restriction functor $\mbox{res}_H^G$ maps 
$^{\perp}{\tt Fib}(kG)$ into $^{\perp}{\tt Fib}(kH)$.
\end{Proposition}

\vspace{-0.05in}

\noindent
{\em Proof.}
(i) Let $M$ be a fibrant $kH$-module; then, $M$ is contained in 
$\mathcal{I}_0(kH)$. Since $H \in \mathfrak{J}$, the diagonal 
$kH$-module $\mbox{Hom}_k(\mbox{res}_H^GB(G),M)$ is injective
and hence the coinduced $kG$-module 
$\mbox{coind}_H^{\, G} \mbox{Hom}_k(\mbox{res}_H^GB(G),M)$ is
injective as well. In view of Lemma A.1, the latter $kG$-module 
is isomorphic with the diagonal $kG$-module 
$\mbox{Hom}_k \! \left( B(G),\mbox{coind}_H^{\, G}M \right)$, 
which is therefore also injective. We conclude that 
$\mbox{coind}_H^{\, G}M \in {\tt Fib}(kG)$.

(ii) Let $M$ be a $kG$-module contained in $^{\perp}{\tt Fib}(kG)$. 
We have to show that the $kH$-module $\mbox{res}_H^GM$ is contained 
in $^{\perp}{\tt Fib}(kH)$. This is an immediate consequence of (i)
though, in view of the Eckmann-Shapiro isomorphism 
$\mbox{Ext}^1_{kH} \! \left( \mbox{res}_H^GM, \_\!\_ \right)
 \simeq \mbox{Ext}^1_{kG} \! \left( M, \mbox{coind}_H^G \_\!\_ \right)$. 
\hfill $\Box$

\begin{Lemma}
The class $\mathfrak{J}$ is subgroup-closed and contains all 
finite groups.
\end{Lemma}

\vspace{-0.05in}

\noindent
{\em Proof.}
Let $H$ be a $\mathfrak{J}$-group and consider a subgroup 
$F \subseteq H$. In order to show that $F \in \mathfrak{J}$, 
consider a group $G$ containing $F$ and a module 
$M \in \mathcal{I}_0(kF)$. Let $\Gamma = G*_FH$ be the 
amalgamated product of $G$ and $H$ along $F$ and consider 
the coinduced $kH$-module 
$N = \mbox{coind}_F^HM \in \mathcal{I}_0(kH)$. Since 
$H \in \mathfrak{J}$, the diagonal $kH$-module 
$\mbox{Hom}_k(\mbox{res}_H^{\Gamma}B(\Gamma),N)$ is injective. 
Restricting to the subgroup $F \subseteq H$, we conclude that 
the $kF$-module 
$\mbox{Hom}_k(\mbox{res}_F^{\Gamma}B(\Gamma),\mbox{res}_F^HN)$ 
is injective as well. We note that the $kF$-module $M$ is a 
direct summand of 
$\mbox{res}_F^H\mbox{coind}_F^HM = \mbox{res}_F^HN$. Moreover,
the $kG$-module $B(G)$ is a direct summand of 
$\mbox{res}_G^{\Gamma}B(\Gamma)$ and hence the restricted 
$kF$-module $\mbox{res}_F^GB(G)$ is a direct summand of 
$\mbox{res}_F^{\Gamma}B(\Gamma)$. It follows that 
$\mbox{Hom}_k(\mbox{res}_F^GB(G),M)$ is a direct summand of
$\mbox{Hom}_k(\mbox{res}_F^{\Gamma}B(\Gamma),\mbox{res}_F^HN)$,
so that $\mbox{Hom}_k(\mbox{res}_F^GB(G),M)$ is an injective 
$kF$-module as well. We have therefore proved that 
$F \in \mathfrak{J}$.

We now let $T$ be a finite group. In order to show that 
$T \in \mathfrak{J}$, consider a group $L$ containing 
$T$ as a subgroup and a module $N \in \mathcal{I}_0(kT)$;
then, $N$ is $k$-injective. Since $B(L)$ is free as a 
$kT$-module (cf.\ \cite{KT}), the diagonal $kT$-module 
$\mbox{Hom}_k \! \left( \mbox{res}_T^LB(L),N \right)$ 
is injective. \hfill $\Box$

\medskip

\noindent
The following result will provide a plethora of 
$\mathfrak{J}$-groups (besides finite groups).

\begin{Proposition}
Let $G$ be a group and $H \subseteq G$ a subgroup contained 
in 
${\scriptstyle{{\bf LH}}}\mathfrak{J} \cup 
 \Phi_{inj}\mathfrak{J}$.
Then, for any $kH$-module $M \in \mathcal{I}_0(kH)$ the 
diagonal $kH$-module 
$\mbox{Hom}_k \! \left( \mbox{res}_H^GB(G),M \right)$ is 
injective.
\end{Proposition}

\vspace{-0.05in}

\noindent
{\em Proof.}
We let $N = \mbox{Hom}_k \! \left( \mbox{res}_H^GB(G),M \right)$
and proceed by distinguishing three cases, according to whether 
(i) $H \in {\scriptstyle{{\bf H}}}\mathfrak{J}$, (ii)
$H \in {\scriptstyle{{\bf LH}}}\mathfrak{J}$ and (iii)
$H \in \Phi_{inj}\mathfrak{J}$.

(i) Assume that $H \in {\scriptstyle{{\bf H}}}\mathfrak{J}$ 
and use induction on the ordinal $\alpha$, which is such that 
$H \in {\scriptstyle{{\bf H}}}_{\alpha}\mathfrak{J}$. If 
$\alpha = 0$, then $H \in \mathfrak{J}$; hence, $N$ is 
injective, in view of the definition of $\mathfrak{J}$-groups. 
Let $\alpha >0$ and assume that the result is known for all
${\scriptstyle{{\bf H}}}_{\beta}\mathfrak{J}$-subgroups 
$K \subseteq G$ and all ordinals $\beta$ with $\beta < \alpha$. 
The group $H$ acts cellularly on a finite dimensional 
contractible CW-complex $X$, with each cell stabilizer 
contained in ${\scriptstyle{{\bf H}}}_{\beta}\mathfrak{J}$ 
for a suitable $\beta < \alpha$. Then, applying the functor
$\mbox{Hom}_k(\_\!\_,N)$ to the cellular chain complex of 
$X$, we obtain a short exact sequence of $kH$-modules 
\[ 0 \longrightarrow N \longrightarrow N^0 
     \longrightarrow N^1 \longrightarrow \cdots 
     \longrightarrow N^d \longrightarrow 0 , \]
where $d$ is the dimension of $X$ and each $N^i$ is a product 
of modules of the form $\mbox{coind}_K^H\mbox{res}_K^HN$, 
for a suitable 
${\scriptstyle{{\bf H}}}_{\beta}\mathfrak{J}$-subgroup 
$K \subseteq H$ for some $\beta < \alpha$. For such a
subgroup $K$ we note that
$\mbox{res}_K^HN = \mbox{Hom}_k \! \left( \mbox{res}_K^GB(G),
 \mbox{res}_K^HM \right)$ and 
$\mbox{res}_K^HM \in \mathcal{I}_0(kK)$. Invoking the 
induction hypothesis, we conclude that 
$N^i \in {\tt Inj}(kH)$ for all $i=0,1, \ldots ,d$ and hence 
$\mbox{id}_{kH} N \leq d$. This inequality holds for any 
$M \in \mathcal{I}_0(kH)$. Since $M$ is the $d$-th cosyzygy 
in an injective resolution of another $kH$-module 
$M' \in \mathcal{I}_0(kH)$, it follows that 
$N = \mbox{Hom}_k \! \left( \mbox{res}_H^GB(G),M \right)$ 
is the $d$-th cosyzygy in an injective resolution of the 
corresponding $kH$-module 
$N' = \mbox{Hom}_k \! \left( \mbox{res}_H^GB(G),M' \right)$.
We have $\mbox{id}_{kH}N' \leq d$ and hence the $kH$-module 
$N$ is injective.

(ii) Assume that $H \in {\scriptstyle{{\bf LH}}}\mathfrak{J}$
and proceed by induction on the cardinality $\kappa$ of $H$.
If $\kappa \leq \aleph_0$, then the
${\scriptstyle{{\bf LH}}}\mathfrak{J}$-group $H$ is actually
contained in ${\scriptstyle{{\bf H}}}\mathfrak{J}$ and we are
done invoking (i) above. If $\kappa$ is uncountable, then we 
may express $H$ as a continuous ascending union of subgroups
$(H_{\lambda})_{\lambda < \kappa}$, each one having cardinality
$< \kappa$. Since the class $\mathfrak{J}$ is subgroup-closed 
(cf.\ Lemma A.3), the class ${\scriptstyle{{\bf LH}}}\mathfrak{J}$
is also subgroup-closed and hence $H_{\lambda}$ is an 
${\scriptstyle{{\bf LH}}}\mathfrak{J}$-group for all $\lambda$.
Then, our induction hypothesis implies that
$\mbox{res}^H_{H_{\lambda}}N \in {\tt Inj}(kH_{\lambda})$
for all $\lambda$. We now consider a $kH$-module $L$ and let
$L_{\lambda} = 
 \mbox{ind}_{H_{\lambda}}^H \mbox{res}^H_{H_{\lambda}}L$
for all $\lambda$. We have 
\begin{equation}
 \mbox{Ext}^n_{kH} \! \left( L_{\lambda},N \right) =
 \mbox{Ext}^n_{kH_{\lambda}} \! \left(
 \mbox{res}^H_{H_{\lambda}}L,\mbox{res}_{H_{\lambda}}^HN
 \right) = 0
\end{equation}
for all $\lambda$ and $n>0$. The ascending family of subgroups
$(H_{\lambda})_{\lambda < \kappa}$ induces a continuous direct
system of $kH$-modules $(L_{\lambda})_{\lambda < \kappa}$ with
surjective structure maps, whose colimit is $L$. The short exact 
sequence of $kH$-modules
\[ 0 \longrightarrow T_{\lambda} \longrightarrow L_0
     \longrightarrow L_{\lambda} \longrightarrow 0 , \]
where $T_{\lambda}$ is the kernel of the structure map
$L_0 \longrightarrow L_{\lambda}$, is the $\lambda$-th term
of a continuous direct system of short exact sequences. The
colimit of the latter direct system of short exact sequences
is the short exact sequence of $kH$-modules
\begin{equation}
 0 \longrightarrow T \longrightarrow L_0
   \longrightarrow L \longrightarrow 0 .
\end{equation}
Here, $T$ is equal to the continuous ascending union of its
submodules $(T_{\lambda})_{\lambda < \kappa}$. Since
$T_{\lambda +1}/T_{\lambda}$ can be identified with
the kernel of the (surjective) structure map
$L_{\lambda} \longrightarrow L_{\lambda +1}$, we conclude
from (2) that
$\mbox{Ext}^n_{kH} \! \left( T_{\lambda +1}/T_{\lambda},N
 \right) =0$
for all $\lambda$ and $n>0$. Then, Eklof's lemma
\cite[Theorem 7.3.4]{EJ2} implies that
$\mbox{Ext}^n_{kH}(T,N) = 0$ for all $n>0$ and the short
exact sequence (3) implies that $\mbox{Ext}^n_{kH}(L,N)=0$
for all $n>1$. Since this holds for any $kH$-module $L$, 
we conclude that $\mbox{id}_{kH}N \leq 1$. The argument 
used at the end of the proof of (i) above applies here as 
well and shows that the $kH$-module $N$ is actually injective.

(iii) Assume now that $H \in \Phi_{inj}\mathfrak{J}$. For 
any $\mathfrak{J}$-subgroup $K \subseteq H$ the restricted 
$kK$-module 
$\mbox{res}_K^HN = \mbox{Hom}_k \! \left( \mbox{res}_K^GB(G),
 \mbox{res}^H_KM \right)$
is injective, in view of the definition of $\mathfrak{J}$-groups,
since $\mbox{res}_K^HM \in \mathcal{I}_0(kK)$. It follows
that $\mbox{id}_{kH}N < \infty$. Since this is the case for any 
$M \in \mathcal{I}_0(kH)$ and the class of these modules is 
closed under products, a standard argument shows that there 
is an upper bound on the injective dimension of $N$ for all 
$M \in \mathcal{I}_0(kH)$. We conclude, as in the last part 
of the proof of (i) above, that the $kH$-module $N$ is actually 
injective. \hfill $\Box$

\medskip

\noindent
We may reformulate Proposition A.4 as follows.

\begin{Theorem}
The class $\mathfrak{J}$ is invariant under the operation 
${\scriptstyle{{\bf LH}}} \cup \Phi_{inj}$.
\end{Theorem}

\vspace{-0.05in}

\noindent
{\em Proof.}
We have to show that 
$\mathfrak{J} = {\scriptstyle{{\bf LH}}}\mathfrak{J} 
 \cup \Phi_{inj}\mathfrak{J}$,
i.e.\ that 
${\scriptstyle{{\bf LH}}}\mathfrak{J} \cup 
 \Phi_{inj}\mathfrak{J} \subseteq \mathfrak{J}$.
Let $H$ be a group contained in 
${\scriptstyle{{\bf LH}}}\mathfrak{J} \cup 
 \Phi_{inj}\mathfrak{J}$. 
In order to prove that $H$ is a $\mathfrak{J}$-group, we 
consider a group $G$ with $H \subseteq G$ and a $kH$-module 
$M \in \mathcal{I}_0(kH)$. We have to show that the diagonal 
$kH$-module $\mbox{Hom}_k(\mbox{res}_H^GB(G),M)$ is injective. 
But this follows from Proposition A.4. \hfill $\Box$

\begin{Corollary}
The group class $\widehat{\mathfrak{F}}(k)$ is a subclass of 
$\mathfrak{J}$.
\end{Corollary}

\vspace{-0.05in}

\noindent
{\em Proof.}
This follows from Lemma A.3 and Theorem A.5. \hfill $\Box$

\begin{Corollary}
Let $H$ be an $\widehat{\mathfrak{F}}(k)$-subgroup 
of a group $G$. Then:

(i) the coinduction functor $\mbox{coind}_H^{\, G}$ 
maps ${\tt Fib}(kH)$ into ${\tt Fib}(kG)$ and 

(ii) the restriction functor $\mbox{res}_H^G$ maps 
$^{\perp}{\tt Fib}(kG)$ into $^{\perp}{\tt Fib}(kH)$.
\end{Corollary}

\vspace{-0.05in}

\noindent
{\em Proof.}
This follows from Proposition A.2, in view of Corollary 
A.6. \hfill $\Box$

\section{Fibrant modules and orthogonality}

\noindent
Having studied in Appendix A the extent to which fibrant 
modules are preserved under coinduction from subgroups, 
we shall now obtain a criterion for a module to be left 
orthogonal to the class of fibrant modules over hierarchically 
defined groups and complete the proof of Theorem 3.3(i), 
which played an important role in Section 3. We fix the 
commutative ring $k$ and follow the approach adopted in 
\cite[$\S $6]{ER1}.

\begin{Proposition}
Let $\mathfrak{C}$ be a subgroup-closed subclass of 
$\widehat{\mathfrak F}(k)$. We consider a group $G$ and 
a $kG$-module $M$, such that
$\mbox{res}_H^GM \in {^{\perp}}{\tt Fib}(kH)$ for any
$\mathfrak{C}$-subgroup $H \subseteq G$. Then, we have:

(i) $\mbox{res}_H^GM \in {^{\perp}}{\tt Fib}(kH)$ for any 
${\scriptstyle{{\bf H}}}\mathfrak{C}$-subgroup $H \subseteq G$,

(ii) $\mbox{res}_H^GM \in {^{\perp}}{\tt Fib}(kH)$ for any
${\scriptstyle{{\bf LH}}}\mathfrak{C}$-subgroup $H \subseteq G$
and

(iii) $\mbox{res}_H^GM \in {^{\perp}}{\tt Fib}(kH)$ for any 
$\Phi_{inj}\mathfrak{C}$-subgroup $H \subseteq G$.
\end{Proposition}
\vspace{-0.05in}
\noindent
{\em Proof.}
(i) We let $H$ be an ${\scriptstyle{{\bf H}}}\mathfrak{C}$-subgroup
of $G$ and use induction on the ordinal number $\alpha$, for which
$H \in {\scriptstyle{{\bf H}}}_{\alpha}\mathfrak{C}$. The
base case $\alpha =0$ is precisely our hypothesis on $M$. For the 
inductive step, assume that $\alpha >0$ and the result is true for 
all ${\scriptstyle{{\bf H}}}_{\beta}\mathfrak{C}$-subgroups of $G$ 
and all $\beta < \alpha$. We consider an
${\scriptstyle{{\bf H}}}_{\alpha}\mathfrak{C}$-subgroup
$H \subseteq G$ and let $N \in {\tt Fib}(kH)$. Then, as
noted in part (i) of the proof of Proposition A.4, there 
exists an exact sequence of $kH$-modules
\begin{equation}
 0 \longrightarrow N \longrightarrow N^0
   \longrightarrow N^1 \longrightarrow \cdots
   \longrightarrow N^d \longrightarrow 0 , 
\end{equation}
where $d$ is the dimension of the $H$-CW-complex witnessing
that $H \in {\scriptstyle{{\bf H}}}_{\alpha}\mathfrak{C}$
and each $N^i$ is a product of $kH$-modules of the form 
$\mbox{coind}_F^H\mbox{res}_F^HN$, with $F$ an
${\scriptstyle{{\bf H}}}_{\beta}\mathfrak{C}$-subgroup
of $H$ for some $\beta < \alpha$. The restriction functor
maps ${\tt Fib}(kH)$ into ${\tt Fib}(kF)$; in particular, 
$\mbox{res}_F^HN \in {\tt Fib}(kF)$. Our induction hypothesis 
implies that $\mbox{res}_F^GM \in {^{\perp}}{\tt Fib}(kF)$ 
and hence
\[ \mbox{Ext}^n_{kH} \! \left( \mbox{res}_H^GM,
   \mbox{coind}_F^H\mbox{res}_F^HN \right) =
   \mbox{Ext}^n_{kF} \! \left( \mbox{res}_F^GM,
   \mbox{res}_F^HN \right) = 0 \]
for such a subgroup $F$ and all $n>0$. It follows that
$\mbox{Ext}_{kH}^n \! \left( \mbox{res}_H^GM,N^i \right) =0$
for all $i$ and all $n>0$. Using dimension shifting, the exact
sequence (5) implies that 
$\mbox{Ext}_{kH}^n \! \left( \mbox{res}_H^GM,N \right) = 0$ if 
$n>d$. Since this holds for any $N \in {\tt Fib}(kH)$, we may
conclude that the $kH$-module $\mbox{res}_H^GM$ is actually contained 
in $^{\perp}{\tt Fib}(kH)$. Indeed, the fibrant $kH$-module $N$ 
is the $d$-th cosyzygy in an injective resolution of another 
fibrant $kH$-module $N'$ and hence
\[ \mbox{Ext}_{kH}^1 \! \left( \mbox{res}_H^GM,N \right) =
   \mbox{Ext}_{kH}^{d+1} \! \left( \mbox{res}_H^GM,N' \right) = 0 . \]
This completes the inductive step of the proof.

(ii) We now consider an ${\scriptstyle{{\bf LH}}}\mathfrak{C}$-subgroup 
$H$ of $G$ and proceed by induction on the cardinality $\kappa$ of $H$.
If $\kappa$ is countable, then the
${\scriptstyle{{\bf LH}}}\mathfrak{C}$-group $H$ is actually contained 
in ${\scriptstyle{{\bf H}}}\mathfrak{C}$ and we are done by (i) above. 
If $\kappa$ is uncountable, then we may express $H$ as a continuous 
ascending union of subgroups $(H_{\lambda})_{\lambda < \kappa}$, each 
one having cardinality $< \kappa$. Since the class
${\scriptstyle{{\bf LH}}}\mathfrak{C}$ is subgroup-closed,
$H_{\lambda}$ is an ${\scriptstyle{{\bf LH}}}\mathfrak{C}$-group and 
our induction hypothesis implies that
$\mbox{res}^G_{H_{\lambda}}M \in {^{\perp}}{\tt Fib}(kH_{\lambda})$
for all $\lambda$. We consider a fibrant $kH$-module $N$ and note 
that $\mbox{res}^H_{H_{\lambda}}N \in {\tt Fib}(kH_{\lambda})$ for 
all $\lambda$. Letting
$M_{\lambda} =
 \mbox{ind}_{H_{\lambda}}^H \mbox{res}^G_{H_{\lambda}}M$,
we have
\begin{equation}
 \mbox{Ext}^1_{kH} \! \left( M_{\lambda},N \right) =
 \mbox{Ext}^1_{kH_{\lambda}} \! \left(
 \mbox{res}^G_{H_{\lambda}}M,\mbox{res}_{H_{\lambda}}^HN
 \right) = 0
\end{equation}
for all $\lambda$. It follows that 
$M_{\lambda} \in {^{\perp}}{\tt Fib}(kH)$ for all $\lambda$. The 
ascending family of subgroups $(H_{\lambda})_{\lambda < \kappa}$ 
induces a direct system of $kH$-modules $(M_{\lambda})_{\lambda < \kappa}$,
whose colimit is $\mbox{res}_H^GM$. Since the left hand class 
$^{\perp}{\tt Fib}(kH)$ of the fibrant cotorsion pair is thick 
(cf.\ \cite[Proposition 4.6]{ER2}), it follows from a result by 
Gillespie \cite[Proposition 3.1]{G} that $^{\perp}{\tt Fib}(kH)$
is closed under direct limits. Hence, we have
$\mbox{res}_H^GM \in {^{\perp}}{\tt Fib}(kH)$, as needed.

(iii) Finally, let $H \subseteq G$ be a 
$\Phi_{inj}\mathfrak{C}$-subgroup. Since the cotorsion pair
$\left( {^{\perp}}{\tt Fib}(kH),{\tt Fib}(kH) \right)$ is
complete, there exists a short exact sequence of $kH$-modules
\[ 0 \longrightarrow N \longrightarrow L
     \longrightarrow \mbox{res}_H^GM \longrightarrow 0 , \]
where $L \in {^{\perp}}{\tt Fib}(kH)$ and $N \in {\tt Fib}(kH)$.
Then, for any $\mathfrak{C}$-subgroup $F \subseteq H$ we have a
short exact sequence of $kF$-modules
\[ 0 \longrightarrow \mbox{res}_F^HN \longrightarrow
     \mbox{res}_F^HL \longrightarrow \mbox{res}_F^GM
     \longrightarrow 0 . \]
Of course, the restricted $kF$-module $\mbox{res}_F^HN$ is fibrant. 
Since $F \in \mathfrak{C} \subseteq \widehat{\mathfrak{F}}(k)$,
Corollary A.7(ii) implies that
$\mbox{res}_F^HL \in {^{\perp}}{\tt Fib}(kF)$. In view of our 
assumption on $M$, we have
$\mbox{res}_F^GM \in {^{\perp}}{\tt Fib}(kF)$. Hence, the thickness
of ${^{\perp}}{\tt Fib}(kF)$ implies that $\mbox{res}_F^HN$ is also 
contained in ${^{\perp}}{\tt Fib}(kF)$. Since
${\tt Fib}(kF) \cap {^{\perp}}{\tt Fib}(kF) = {\tt Inj}(kF)$, it 
follows that the $kF$-module $\mbox{res}_F^HN$ is injective. This 
is the case for any $\mathfrak{C}$-subgroup $F$ of the
$\Phi_{inj}\mathfrak{C}$-group $H$ and hence
$\mbox{id}_{kH}N < \infty$. It follows from 
\cite[Proposition 4.6]{ER2} that $N \in {^{\perp}}{\tt Fib}(kH)$ 
and hence the thickness of ${^{\perp}}{\tt Fib}(kH)$ implies that
$\mbox{res}_H^GM \in {^{\perp}}{\tt Fib}(kH)$, as needed. \hfill $\Box$

\medskip

\noindent
A transfinite iteration of the result in Proposition B.1 has the 
following consequence.

\begin{Corollary}
The following conditions are equivalent for a group $G$ and 
a $kG$-module $M$:

(i) $\mbox{res}_H^GM \in {^{\perp}}{\tt Fib}(kH)$ for any
finite subgroup $H \subseteq G$.

(ii) $\mbox{res}_H^GM \in {^{\perp}}{\tt Fib}(kH)$ for any
$\widehat{\mathfrak{F}}(k)$-subgroup $H \subseteq G$.
\end{Corollary}

\vspace{-0.05in}

\noindent
{\em Proof.}
(i)$\rightarrow$(ii): We recall from $\S $1.I the definition 
of the group class $\widehat{\mathfrak{F}}(k)$. There is a 
hierarchy of subgroup-closed classes ${\mathfrak F}'(k)_{\alpha}$, 
where $\alpha$ is any ordinal number, which are defined as follows: 

(a) $\mathfrak{F}'(k)_0 = \mathfrak{F}$ is the class of finite 
    groups,

(b) ${\mathfrak F}'(k)_{\alpha +1} = 
     {\scriptstyle{{\bf LH}}}{\mathfrak F}'(k)_{\alpha} 
     \cup \Phi_{inj} {\mathfrak F}'(k)_{\alpha}$
    for any ordinal $\alpha$ and 

(c) $\mathfrak{F}'(k)_{\alpha} = 
     \bigcup_{\beta < \alpha}\mathfrak{F}'(k)_{\beta}$
    for any limit ordinal $\alpha$. 
\newline
Then, $\widehat{\mathfrak{F}}(k)$ is the class consisting of those 
groups which are contained in ${\mathfrak F}'(k)_{\alpha}$, for 
some ordinal $\alpha$. Hence, (ii) will follow if we show that 
$\mbox{res}_H^GM \in {^{\perp}}{\tt Fib}(kH)$ for any
${\mathfrak F}'(k)_{\alpha}$-subgroup $H \subseteq G$ for any 
ordinal $\alpha$. The latter claim is proved using transfinite 
induction: The claim for $\alpha =0$ is precisely condition (i), 
whereas Proposition B.1 (applied to the case where the 
subgroup-closed class $\mathfrak{C}$ therein is 
${\mathfrak F}'(k)_{\alpha}$) provides the inductive step.

(ii)$\rightarrow$(i): This is obvious, since
$\mathfrak{F} \subseteq \widehat{\mathfrak{F}}(k)$. \hfill $\Box$

\medskip

\noindent
We can finally prove Theorem 3.3(i).

\medskip

\noindent
{\em Proof of Theorem 3.3(i).}
Assume that $M \in {^{\perp}}{\tt Fib}(kG)$. Since finite groups 
are contained in $\widehat{\mathfrak{F}}(k)$, Corollary A.7(ii) 
implies that $\mbox{res}_H^GM \in {^{\perp}}{\tt Fib}(kH)$ for 
any finite subgroup $H \subseteq G$.\footnote{In fact, since 
$G$ is an $\widehat{\mathfrak{F}}(k)$-group and
$\widehat{\mathfrak{F}}(k)$ is subgroup-closed, it follows from 
Corollary A.7(ii) that $\mbox{res}_H^GM \in {^{\perp}}{\tt Fib}(kH)$
for any subgroup $H \subseteq G$.}

Conversely, assume that $\mbox{res}_H^GM \in {^{\perp}}{\tt Fib}(kH)$
for any finite subgroup $H \subseteq G$. Corollary B.2 then implies
that $\mbox{res}_H^GM \in {^{\perp}}{\tt Fib}(kH)$ for any 
$\widehat{\mathfrak{F}}(k)$-subgroup $H \subseteq G$. Since 
$G \in \widehat{\mathfrak{F}}(k)$, $G$ is an 
$\widehat{\mathfrak{F}}(k)$-subgroup of itself and hence 
$M \in {^{\perp}}{\tt Fib}(kG)$. \hfill $\Box$

\medskip

{\small {\sc Department of Mathematics,
             University of Athens,
             Athens 15784,
             Greece}}

{\em E-mail addresses:} {\tt emmanoui@math.uoa.gr} and
                        {\tt otalelli@math.uoa.gr}

\end{document}